\documentclass{amsart}

\usepackage{amsmath,amssymb,amsthm}
\usepackage{braket}
\usepackage{color}
\usepackage{mathrsfs}
\input{Dynkin.sty}

\newtheorem{theorem}{Theorem}
\newtheorem{definition}[theorem]{Definition}
\newtheorem{prop}[theorem]{Proposition}
\newtheorem{lemma}[theorem]{Lemma}
\newtheorem{cor}[theorem]{Corollary}

\newtheorem{rem}[theorem]{Remark}

\newtheorem*{thmA}{Theorem A}
\newtheorem*{thmB}{Theorem B}
\newtheorem*{thmC}{Theorem C}

\numberwithin{theorem}{section}
\numberwithin{equation}{section}
\newcommand{\W}{\mathcal{W}}
\newcommand{\Z}{\mathbb{Z}}
\newcommand{\C}{\mathbb{C}}
\newcommand{\g}{\mathfrak{g}}
\newcommand{\h}{\mathfrak{h}}
\newcommand{\z}{\mathfrak{z}}
\newcommand{\slf}{\mathfrak{sl}}
\newcommand{\gl}{\mathfrak{gl}}
\newcommand{\osp}{\mathfrak{osp}}
\newcommand{\spf}{\mathfrak{sp}}

\newcommand{\ad}{\operatorname{ad}}

\newcommand{\la}{\lambda}
\newcommand{\der}{\partial}
\newcommand{\e}{\mathrm{e}}
\newcommand{\Ker}{\mathrm{Ker}}
\newcommand{\Img}{\mathrm{Im}}

\newcommand{\Id}{\mathrm{Id}}
\newcommand{\dst}{d_{\mathrm{st}}}
\newcommand{\dne}{d_{\mathrm{ne}}}
\newcommand{\dch}{d_{\chi}}
\newcommand{\dstz}{{\dst}_{(0)}}
\newcommand{\dnez}{{\dne}_{(0)}}
\newcommand{\dchz}{{\dch}_{(0)}}
\newcommand{\bJ}{\bar{J}}
\newcommand{\bPhi}{\bar{\Phi}}
\newcommand{\tS}{\widetilde{S}}
\newcommand{\td}{\widetilde{d}}
\newcommand{\tE}{\widetilde{E}}
\newcommand{\tmu}{\widetilde{\mu}}

\newcommand{\pC}{C'}
\newcommand{\hC}{\widehat{C}}
\newcommand{\hpC}{\widehat{C}'}
\newcommand{\Cinf}{\mathcal{C}}
\newcommand{\res}{\mathrm{res}}
\newcommand{\inv}{(\hspace{0.5mm}\cdot\hspace{0.5mm}|\hspace{0.5mm}\cdot\hspace{0.5mm})}
\newcommand{\foral}{\mathrm{for}\ \mathrm{all}}
\newcommand{\conf}{\Delta}
\newcommand{\p}{{\epsilon}}
\newcommand{\X}{y}
\newcommand{\kh}{\nu}

\newcommand{\str}{\mathrm{str}}
\newcommand{\Vtau}{V^{\tau_{k}}(\g_{0})}
\newcommand{\Vtaun}{V^{\tau_{k}}(\g_{\leq0})}
\newcommand{\Fch}{F^{\mathrm{ch}}(\g_{+})}
\newcommand{\Fne}{F(\g_{\frac{1}{2}})}
\newcommand{\Hi}{\mathcal{H}}
\newcommand{\F}{\mathcal{F}}
\newcommand{\V}{\mathcal{V}}
\newcommand{\E}{\mathcal{E}}
\newcommand{\Po}{\mathcal{P}}
\newcommand{\Del}{\Delta^{\Gamma}}
\newcommand{\Dp}{\Delta_{>0}}
\newcommand{\Dn}{\Delta_{\frac{1}{2}}}
\newcommand{\Pp}{\Pi^{\Gamma}}
\newcommand{\JS}{S(\g_{0}[t^{-1}]t^{-1})}
\newcommand{\JSn}{S(\g_{\leq0}[t^{-1}]t^{-1})}
\newcommand{\palpha}{p(\alpha)}
\newcommand{\pbeta}{p(\beta)}
\newcommand{\pgamma}{p(\gamma)}
\newcommand{\ealpha}{e_{\alpha}}
\newcommand{\ebeta}{e_{\beta}}
\newcommand{\egamma}{e_{\gamma}}

\title{Screening Operators for $\W$-algebras}
\author{Naoki Genra}
\address{Research Institute for Mathematical Sciences, Kyoto University,
 Kyoto 606-8502 JAPAN}
\email{gnr@kurims.kyoto-u.ac.jp}

\begin{document}
\maketitle

\begin{abstract}
Let $\g$ be a simple finite-dimensional Lie superalgebra with a non-degenerate supersymmetric even invariant bilinear form, $f$ a nilpotent element in the even part of $\g$, $\Gamma$ a good grading of $\g$ for $f$ and $\W^{k}(\g,f;\Gamma)$ the (affine) $\W$-algebra associated with $\g,f,k,\Gamma$ defined by the generalized Drinfeld-Sokolov reduction. In this paper, we present each $\W$-algebra as the intersection of kernels of the screening operators, acting on the tensor vertex superalgebra of an affine vertex superalgebra and a neutral free superfermion vertex superalgebra. As applications, we prove that the $\W$-algebra associated with a regular nilpotent element in $\osp(1,2n)$ is isomorphic to the $\W B_{n}$-algebra introduced by Fateev and Lukyanov, and that the $\W$-algebra associated with a subregular nilpotent element in $\slf_{n}$ is isomorphic to the $\W^{(2)}_{n}$-algebra introduced by Feigin and Semikhatov.
\end{abstract}

\section{Introduction}\label{Introduction sec}

Let $\g$ be a simple finite-dimensional Lie superalgebra with a non-degenerate supersymmetric even invariant bilinear form $\inv$, $f$ be a nilpotent element in the even part of  $\g$, $\Gamma$ be a good grading of $\g$ for $f$, denoted by
\begin{align*}
\Gamma:\g=\bigoplus_{j\in\frac{1}{2}\Z}\g_{j}
\end{align*}
and $\W^{k}(\g,f;\Gamma)$ be the (affine) $\W$-algebra defined as the BRST cohomology associated with $\g,f,k$ and $\Gamma$ (\cite{FF90,KRW}). The $\W$-algebra is a $\frac{1}{2}\Z_{\geq0}$-graded vertex superalgebra and it is conformal if its level is not critical i.e. $k\neq-h^{\vee}$, where $h^{\vee}$ is the dual Coxeter number of $\g$. For fixed $\g$, $f$ and $k$, the vertex superalgebra structure of the $\W$-algebra $\W^{k}(\g,f;\Gamma)$ does not depend on the choice of the good grading $\Gamma$, although the conformal grading does; see \cite{AKM, BG}.

In this paper, we present an arbitrary $\W$-algebra as the intersection of kernels of the screening operators. In the special case that the good grading $\Gamma$ can be chosen such that $\g_{0}$ is a Cartan subalgebra, our presentation gives a free field realization in terms of the screening operators, which recovers a result of Feigin and Frenkel \cite{FF92} in the case that $\g$ is a simple Lie algebra and $f$ is a regular nilpotent element. As applications, we prove that the $\W$-algebra associated with a regular nilpotent element in $\osp(1,2n)$ is isomorphic to the $\W B_{n}$-algebra introduced by Fateev and Lukyanov, and that the $\W$-algebra associated with a subregular nilpotent element in $\slf_{n}$ is isomorphic to the $\W^{(2)}_{n}$-algebra introduced by Feigin and Semikhatov. These resluts were conjectured in \cite{IMP, FOR, Wat} and \cite{ACGHR}, respectively.

\smallskip
To clarify our results, let $\Delta$ be a set of roots of $\g$ and $\Pi$ be a set of simple roots compatible with $\Gamma$. For $j\in\frac{1}{2}$, we denote by $\Delta_{j}$ the set of roots whose root vector lies in $\g_{j}$ and set $\Pi_{j}=\Pi\cap\Delta_{j}$. Then, there is a decomposition $\Pi=\Pi_{0}\sqcup\Pi_{\frac{1}{2}}\sqcup\Pi_{1}$; see e.g. \cite{EK, Ho}. Denote by $\Del=\Delta\backslash\Delta_{0}$ the restricted root system associated with $\Gamma$ (\cite{BG}) and by
\begin{align*}
\Pp=\{\alpha\in\Dp\mid\alpha\ \mathrm{is}\ \mathrm{indecomposable}\ \mathrm{in}\ \Dp\}
\end{align*}
the base of $\Del$. Then, $\Pp=\Pp_{\frac{1}{2}}\sqcup\Pp_{1}$, where $\Pp_{j}=\Pp\cap\Delta_{j}$. Let $Q_{0}$ be a root lattice of $\g_{0}$. The equivalence relation on $\Pp$ defined by
\begin{align*}
\alpha\sim\beta\iff\alpha-\beta\in Q_{0}
\end{align*}
for $\alpha,\beta\in\Pp$ gives the quotient set $[\Pp]$ of $\Pp$. Let $[\beta]$ be the equivalent class of $\beta\in\Pp$ in $[\Pp]$. Since the equivalence relation on $\Pp$ is homogeneous with respect to the grading, the decomposition of $\Pp$ induces that of $[\Pp]$:
\begin{align*}
[\Pp]=[\Pp_{\frac{1}{2}}]\sqcup[\Pp_{1}].
\end{align*}

Let $\Vtau$ be an affine vertex superalgebra associated with the subalgebra $\g_{0}$ of $\g$ with degree $0$ and its invariant bilinear form $\tau_{k}$ (see \eqref{tauk eq} for the definition), and $\Fne$ be a neutral free superfermion vertex superalgebra associated with the subspace $\g_{\frac{1}{2}}$ of $\g$ with degree $\frac{1}{2}$ (\cite{KRW}).

\begin{thmA}[Theorem \ref{main1 thm}]
For generic $k$, the $\W$-algebra $\W^{k}(\g,f;\Gamma)$ is isomorphic to the vertex subalgebra of $\Vtau\otimes\Fne$, which is the intersection of kernels of the screening operators $Q_{[\beta]}$ for all $[\beta]\in[\Pp]$:
\begin{align}\label{intro1 thm}
\W^{k}(\g,f;\Gamma)\simeq\bigcap_{[\beta]\in[\Pp]}\Ker\ Q_{[\beta]}.
\end{align}
\end{thmA}

See \eqref{Screening eq1} and \eqref{Screening eq2} for the definitions of the screening operators $Q_{[\beta]}$.

We note that Theorem A describes the image of the Miura map defined in \cite{KW2}, see Section \ref{Miura sec}.

\smallskip
In the special case that $\g_{0}$ is a Cartan subalgebra $\h$, the vertex superalgebra $\Vtau$ is isomorphic to the Heisenberg vertex algebra $\Hi$ associated with $\h$. Moreover, since the equivalence relation on $\Pp$ is trivial and $\Delta_{0}$ is empty, the set $[\Pp]$ is equal to $\Pi$. Let $\chi$ be a linear function on $\g$ defined by $\chi(u)=(f|u)$ for all $u\in\g$, $\alpha(z)$ be the vertex operator on $\Hi$ associated with $\alpha\in\Pi$, where we identify the dual space $\h^{*}$ with $\h$ via the fixed bilinear form, and $\Phi_{\beta}(z)$ be the vertex operator on $\Fne$ associated with $\beta\in\Pi_{\frac{1}{2}}$. Then we obtain a more explicit presentation of the $\W$-algebra.

\begin{thmB}[Theorem \ref{main2 thm}]
Assume that $\Gamma$ can be chosen so that $\g_{0}$ is a Cartan subalgebra $\h$ of $\g$. Then, for generic $k$, the $\W$-algebra $\W^{k}(\g,f;\Gamma)$ is isomorphic to the vertex subalgebra of $\Hi\otimes\Fne$, which is the intersection of kernels of the screening operators:
\begin{align*}
\W^{k}(\g,f;\Gamma)\simeq\bigcap_{\begin{subarray}{c}\alpha\in\Pi_{1}\\ \chi(e_{\alpha})\neq0\end{subarray}}\Ker\int\e^{-\frac{1}{\kh}\int\alpha(z)}\ dz\cap\bigcap_{\alpha\in\Pi_{\frac{1}{2}}}\Ker\int:\e^{-\frac{1}{\kh}\int\alpha(z)}\Phi_{\alpha}(z):dz,
\end{align*}
where $\kh=\sqrt{\mathstrut k+h^{\vee}}$.
\end{thmB}

Theorem B for a simple Lie algebra and $f$ a regular nilpotent element was known by Feigin and Frenkel \cite{FF92}, and for a simple Lie superalgebra and $f$ a superprincipal\footnote{There exists an odd nilpotent element $F\in\g_{-\frac{1}{2}}$ with $[F,F]=f$ ($f$ being regular nilpotent) and these  two vectors form part of $\osp(1,2)\subset\g$.} nilpotent element by Heluani and Dias \cite{HD} (see Remark 3.4 of \cite{HD}).

\smallskip
As applications of Theorem A and B, we have the following results that were conjectured in \cite{IMP, FOR, Wat} and \cite{ACGHR}.

\begin{thmC}[Theorem \ref{WBn thm} and Theorem \ref{W2n thm}]\ \\
\vspace{-5mm}
\begin{enumerate}
\item The $\W$-algebra associated with a regular nilpotent even element $f_{reg}$ in $\osp(1,2n)$ at level $k\neq-n-\frac{1}{2}$ is isomorphic to the Fateev-Lukyanov $\W B_{n}$-algebra (\cite{FL}):
\begin{align*}
\W^{k}(\osp(1,2n),f_{reg};\Gamma)\simeq\W B_{n}.
\end{align*}
\item The $\W$-algebra associated with a subregular nilpotent element $f_{sub}$ in $\slf_{n}$ at level $k\neq-n$ is isomorphic to the Feigin-Semikhatov $\W^{(2)}_{n}$-algebra (\cite{FS}):
\begin{align*}
\W^{k}(\slf_{n},f_{sub};\Gamma)\simeq\W^{(2)}_{n}.
\end{align*}
\end{enumerate}
\end{thmC}

For $n=1$, Theorem C (1) was shown in \cite{KRW} as the $\W B_{1}$-algebra is the Neveu-Schwarz vertex superalgebra (or the Super Virasoro vertex superalgebra) (\cite{FL}). We apply Theorem B to $\osp(1,2n)$ and $f_{reg}$ to prove Theorem C (1) for generic $k$, and use a result related to the Miura map to extend the proof to non-critical level $k$, see the proof of Theorem \ref{WBn thm} for the details.

For $n=3$, Theorem C (2) follows from results of \cite{KRW, FS}, for the $\W^{(2)}_{3}$-algebra is the Bershadsky-Polyakov vertex algebra (\cite{B,P}). We apply Theorem A to $\g=\slf_{n}$ and $f=f_{sub}$, and prove Theorem C (2) using the Wakimoto construction of the affine vertex algebra of $\slf_{2}$ for generic $k$ and a result related to the Miura map for non-critical level $k$, see the proof of Theorem \ref{W2n thm}.

\smallskip
The paper is organized as follows. In Section \ref{W-def subsec}, we review the definition of the $\W$-algebra $\W^{k}(\g,f;\Gamma)$. In Section \ref{KW subsec}, we recall the definition of the complex $C_{k}$ that is used to compute $\W^{k}(\g,f;\Gamma)$. In Section \ref{classical subsec}, we consider the classical limit of $C_{k}$, which is used to prove Theorem A. In Section \ref{Screening sec}, we set up some notations to state our main results. In particular, we define the screening operators, which consist of the vertex operators on $\Fne$ and certain operators $S^{\alpha}(z)$. In Section \ref{main thm sec}, we deduce Theorem B from Theorem A. Section \ref{proof sec} is devoted to the proof of Theorem A. We define the weight filtration on $C_{k}$ in Section \ref{weight subsec} and study the first spectral sequence $E_{1}$ in Section \ref{E1 subsec} by using the classical limit. The key fact that we use to analyze $E_{1}$ is Lemma \ref{Cinf lemma}, which is proved in Section \ref{Cinf subsec}. In Section \ref{VAstr sec}, we give another construction of $S^{\alpha}(z)$ and derive some results from it. In Section \ref{proof subsec}, we prove Theorem A. Section \ref{Miura sec} is devoted to the study of the Miura map. In Section \ref{application sec}, we prove Theorem C.

\vspace{3mm}

{\it Acknowledgments}\quad This paper is the master thesis of the author and he wishes to express his gratitude to his supervisor Tomoyuki Arakawa for suggesting the problems and lots of advice to improve this paper. He thanks Hiroshi Yamauchi and Kazuya Kawasetsu for useful comments and discussions. He is deeply grateful to Toshiro Kuwabara for explaining him to the geometric aspects of the $\W^{(2)}_{n}$-algebra and giving him the crucial insight to prove Theorem C (2).

\section{Affine $\W$-algebras}

\subsection{Definitions}\label{W-def subsec}
In this section we review the definition of the (affine) $\W$-algebras via the generalized Drinfeld-Sokolov reduction given by Kac, Roan and Wakimoto \cite{KRW}. Let $\g$ be a simple Lie superalgebra with a non-degenerate even supersymmetric invariant bilinear form $\inv$, $f$ be a nilpotent element of the even part of $\g$, $k$ be a complex number and $\Gamma$ be a good grading of $\g$ with respect to $f$, where a $\frac{1}{2}\Z$-grading
\begin{align*}
\Gamma:\g=\bigoplus_{j\in\frac{1}{2}\Z}\g_{j}
\end{align*}
is called good if $f\in\g_{-1}$ and $\ad f:\g_{j}\rightarrow\g_{j-1}$ is injective for $j\geq\frac{1}{2}$ and surjective for $j\leq\frac{1}{2}$. Let $x$ be a semisimple element in $\g$ such that $\g_{j}=\{u\in\g\mid[x,u]=j u\}$. Choose a Cartan subalgebra $\h$ containing $x$, a basis $\{e_{i}\}_{i\in I}$ of $\h$, where $I=\{1,\ldots,\mathrm{rank}\ \g\}$, the root system $\Delta$ of $(\g,\h)$ and a set of simple roots $\Pi$ such that $\ealpha\in\g_{\geq0}$ for all positive roots $\alpha$, where $e_{\alpha}$ is a non-zero root vector. We normalize the invariant bilinear form $\inv$ on $\g$ by $(\theta|\theta)=2$, where $\theta$ is a highest root of the even part of $\g$. Let $\Delta_{j}=\{\alpha\in\Delta\mid\ealpha\in\g_{j}\}$ and set $\Pi_{j}=\Pi\cap\Delta_{j}$ for all $j\in\frac{1}{2}\Z$, $\chi$ be a linear function on $\g$ defined by $\chi(u)=(f|u)$ for $u\in\g$, $\palpha\in\Z/2\Z$ be the parity of $\ealpha$ for all $\alpha\in I\cup\Delta$ and $c_{\alpha,\beta}^{\gamma}\in\C$ be a structure constant for all $\alpha,\beta,\gamma\in I\cup\Delta$ such that $[\ealpha,\ebeta]=\sum_{\gamma\in I\cup\Delta}c_{\alpha,\beta}^{\gamma}\egamma$. Denote by $\kappa_{\g}$ the Killing form of $\g$ and by $\str_{\g}$ the supertrace on $\g$.

Given a vertex superalgebra $V$ (see \cite{FBZ,K98} for the definitions), we denote by $
Y(A,z)=A(z)=\sum_{n\in\Z}A_{(n)}z^{-n-1}$ the vertex operator for $A\in V$ and by $[A_{\la}B]=\sum_{n=0}^{\infty}(A_{(n)}B)\frac{\la}{n!}$ the $\la$-bracket of $A,B\in V$.

Consider the vertex superalgebra
\begin{align*}
C^{k}(\g,f;\Gamma)=V^{k}(\g)\otimes\Fch\otimes\Fne,
\end{align*}
where $V^{k}(\g)$ is the universal affine vertex superalgebra associated with $\g$, $\Fch$ is the charged free superfermion vertex superalgebra associated with $\g_{+}\oplus\g_{+}^{*}$ with reversed parities and $\Fne$ is the neutral free superfermion vertex superalgebra associated with $\g_{\frac{1}{2}}$ (see \cite{KRW} for definitions), where
\begin{align*}
\g_{+}=\bigoplus_{j>0}\g_{j}.
\end{align*}
Denote by $u(z)$ the vertex operator on $V^{k}(\g)$ associated with $u\in\g$, by $\varphi_{\alpha}(z),\varphi^{\alpha}(z)$ the vertex operators on $\Fch$ associated with $\ealpha, \ealpha^{*}$ for $\alpha\in\Dp$, by $\Phi_{\alpha}(z)$ the vertex operator on $\Fne$ associated with $\alpha\in\Dn$. Note that
\begin{align*}
&[u_{\la}v]=[u,v]+k(u|v)\la,\quad
[{\varphi_{\alpha}}_{\la}\varphi^{\alpha'}]=\delta_{\alpha,\alpha'},\\
&[{\varphi_{\alpha}}_{\la}\varphi_{\alpha'}]=[{\varphi^{\alpha}}_{\la}\varphi^{\alpha'}]=0,\quad
[{\Phi_{\beta}}_{\la}\Phi_{\beta'}]=(f|[e_{\beta},e_{\beta'}])
\end{align*}
for all $u,v\in\g$, $\alpha,\alpha'\in\Dp$ and $\beta,\beta'\in\Dn$. Using the charges in $\Fch$ defined by $\mathit{charge}(\varphi_{\alpha})=-1$ and $\mathit{charge}(\varphi^{\alpha})=1$, one has the induced charge decomposition of $C^{k}(\g,f;\Gamma)$. Let
\begin{align*}
d(z)=\dst(z)+\dne(z)+\dch(z),
\end{align*}
where
\begin{align*}
&\dst(z)=\sum_{\alpha\in\Dp}(-1)^{\palpha}:e_{\alpha}(z)\varphi^{\alpha}(z):
-\frac{1}{2}\sum_{\alpha,\beta,\gamma\in\Dp}(-1)^{\palpha\pgamma}c_{\alpha,\beta}^{\gamma}:\varphi_{\gamma}(z)\varphi^{\alpha}(z)\varphi^{\beta}(z):,\\
&\dne(z)=\sum_{\alpha\in\Dn}:\varphi^{\alpha}(z)\Phi_{\alpha}(z):,\quad
\dch(z)=\sum_{\alpha\in\Dp}\chi(e_{\alpha})\varphi^{\alpha}(z).
\end{align*}
Then $d_{(0)}^{2}=0$ and $d_{(0)}$ is a differential compatible with the charge decomposition of $C^{k}(\g,f;\Gamma)$. Thus, $(C^{k}(\g,f;\Gamma),d_{(0)})$ is a cohomology complex. We denote the cohomology of this complex by
\begin{align}
\label{W-alg eq}
\W^{k}(\g,f;\Gamma)=H^{\bullet}(C^{k}(\g,f;\Gamma),d_{(0)}),
\end{align}
which inherits a vertex superalgebra structure from $C^{k}(\g,f;\Gamma)$. Since the $j$-th comohology of \eqref{W-alg eq} is zero if $j\neq0$ \cite{KW2, KW3}, $\W^{k}(\g,f;\Gamma)$ is the $0$-th cohomology of \eqref{W-alg eq}. This vertex superalgebra is called the {\it $\W$-algebra associated with $\g$, $f$, $k$, and $\Gamma$}.  We note that the vertex superalgebra structure of the $\W$-algebra does not depend on the choice of the good grading $\Gamma$ \cite{BG,AKM} (but the conformal grading of the $\W$-algebra depends on $\Gamma$).

\subsection{Decomposition of complex}\label{KW subsec}
We have \cite{KW2,KW3} the decomposition of the complex $C^{k}(\g,f,\Gamma)=C^{-}\otimes C_{k}$ such that $H(C^{-})=\C$ and $C_{k}$ is a vertex superalgebra generated by $J^{u}(z)$ for all $u\in\g_{\leq0}$, $\varphi^{\alpha}(z)$ for all $\alpha\in\Dp$ and $\Phi_{\alpha}(z)$ for all $\alpha\in\Dn$, where
\begin{align*}
J^{u}(z)=u(z)+\sum_{\alpha,\beta\in\Dp}(-1)^{\palpha}c_{u,\beta}^{\alpha}:\varphi_{\alpha}(z)\varphi^{\beta}(z):
\end{align*}
for $u\in\g_{\leq0}$. Note that
\begin{align*}
[{J^{u}}_{\la}J^{v}]=J^{[u,v]}+\tau_{k}(u|v)\la,\quad[{\varphi^{\alpha}}_{\la}J^{u}]=\sum_{\beta\in\Dp}c_{u,\beta}^{\alpha}\varphi^{\beta},\quad[{\Phi_{\gamma}}_{\la}J^{u}]=0
\end{align*}
for all $u,v\in\g_{\leq0}$, $\alpha\in\Dp$, $\gamma\in\Dn$, where
\begin{align}
\label{tauk eq}
\tau_{k}(u|v)=k(u|v)+\frac{1}{2}\kappa_{\g}(u|v)-\frac{1}{2}\kappa_{\g_{0}}(u|v).
\end{align}
The vertex superalgebra $C_{k}$ is a subcomplex of $C^{k}(\g,f;\Gamma)$ and we have
\begin{align*}
&[{\dst}_{\la}J^{u}]=-\sum_{\begin{subarray}{c}\beta\in\Dp\\ \alpha\in I\cup\Delta_{\leq 0}\end{subarray}}(-1)^{\palpha}c_{u,\beta}^{\alpha}:J^{\ealpha}\varphi^{\beta}:+\sum_{\beta\in\Dp}(\mathit{a}_{k}(u|\ebeta)\der+\mathit{b}_{k}(u|\ebeta)\la)\varphi^{\beta},\\
&[{\dne}_{\la}J^{u}]=\sum_{\begin{subarray}{c}\beta\in\Dp\\ \alpha\in\Dn\end{subarray}}c_{u,\beta}^{\alpha}:\Phi_{\alpha}\varphi^{\beta}:,\quad
[{\dch}_{\la}J^{u}]=\sum_{\beta\in\Dp}\chi([u,\ebeta])\varphi^{\beta},\\
&[{\dst}_{\la}\varphi^{\alpha}]=-\frac{1}{2}\sum_{\beta,\gamma\in\Dp}(-1)^{\palpha\pbeta}c_{\beta,\gamma}^{\alpha}:\varphi^{\beta}\varphi^{\gamma}:,\quad
[{\dne}_{\la}\Phi_{\alpha}]=\sum_{\beta\in\Dn}\chi([\ebeta,\ealpha])\varphi^{\beta},\\
&[{\dne}_{\la}\varphi^{\alpha}]=[{\dch}_{\la}\varphi^{\alpha}]=[{\dst}_{\la}\Phi_{\alpha}]=[{\dch}_{\la}\Phi_{\alpha}]=0,
\end{align*}
where
\begin{align*}
&\mathit{a}_{k}(v|w)=\str_{\g}((\ad v)p_{+}(\ad w))+k(v|w),\\
&\mathit{b}_{k}(v|w)=\str_{\g}(p_{+}(\ad v)(\ad w))+k(v|w)
\end{align*}
for $v,w\in\g$ and $p_{+}:\g\rightarrow\g_{+}$ is a projection map. Since $H(C^{-})=\C$, $H(C^{k}(\g,f;\Gamma))=H(C_{k})$. Moreover, $H^{n}(C_{k})=0$ for all $n\neq0$, see \cite{KW2,KW3}. Therefore
\begin{align*}
\W^{k}(\g,f;\Gamma)=H(C_{k})=H^{0}(C_{k}).
\end{align*}

Let $\pC_{k}$ be a vertex subalgebra of $C_{k}$ generated by $J^{u}(z)$ for all $u\in\g_{\leq0}$ and $\varphi^{\alpha}(z)$ for all $\alpha\in\Dp$. Since ${\dst}_{(0)}^{2}=0$ and by the above formulas, the vertex superalgebra $\pC_{k}$ is a complex with the differential ${\dst}_{(0)}$. We have
\begin{align}
\label{pCk eq}
C_{k}=\pC_{k}\otimes\Fne.
\end{align}

Let $\Vtaun$ be a vertex subalgebra of $\pC_{k}$ generated by $J^{u}(z)$ for all $u\in\g_{\leq0}$ and $\Vtau$ be a vertex subalgebra of $\Vtaun$ generated by $J^{u}(z)$ for all $u\in\g_{0}$. Then
\begin{align}
\label{Vtaun eq}C_{k}^{(0)}=\Vtaun\otimes\Fne,\\
\label{dst0 eq}{\dst}_{(0)}(\Vtau\otimes\Fne)=0.
\end{align}

\subsection{Classical Limit}\label{classical subsec}
Provided that $k+h^{\vee}\neq0$, define
\begin{align*}
\p=\frac{1}{k+h^{\vee}},\quad
\bJ^{u}(z)=\p J^{u}(z),\quad
\bPhi_{\alpha}(z)=\p\Phi_{\alpha}(z)
\end{align*}
for $u\in\g_{\leq0}$ and $\alpha\in\Dn$.

\begin{lemma}
\label{qcal lemma}
Suppose $k+h^{\vee}\neq0$ and replace $f$ by $\p^{-1}f$. Then
\begin{align*}
&[{\bJ^{u}}_{\la}\bJ^{v}]=\p(\bJ^{[u,v]}+\bar{\tau}_{\p}(u|v)\la),\quad
[{\varphi^{\alpha}}_{\la}\bJ^{u}]=\p\sum_{\beta\in\Dp}c_{u,\beta}^{\alpha}\varphi^{\beta},\\
&[{\bPhi}_{\alpha\hspace{0.1mm}\la}\bPhi_{\beta}]=\p\cdot\chi([e_{\alpha},e_{\beta}]),\hspace{10mm}
[{\bJ^{u}}_{\la}{\bPhi}_{\alpha}]=[{\varphi^{\alpha}}_{\la}\varphi^{\beta}]=[{\varphi^{\alpha}}_{\la}\bPhi_{\beta}]=0,
\end{align*}
\begin{align*}
[{\dst}_{\la}\bJ^{u}]=-\sum_{\begin{subarray}{c}\beta\in\Dp\\ \alpha\in I\cup\Delta_{\leq 0}\end{subarray}}(-1)^{\palpha}c_{u,\beta}^{\alpha}:\bJ^{\ealpha}\varphi^{\beta}:+\sum_{\beta\in\Dp}(\bar{\mathit{a}}_{\p}(u|\ebeta)\der+\bar{\mathit{b}}_{\p}(u|\ebeta)\la)\varphi^{\beta},
\end{align*}
\begin{align*}
[{\dne}_{\la}\bJ^{u}]=\sum_{\begin{subarray}{c}\beta\in\Dp\\ \alpha\in\Dn\end{subarray}}c_{u,\beta}^{\alpha}:\bPhi_{\alpha}\varphi^{\beta}:,\hspace{29mm}
&[{\dch}_{\la}\bJ^{u}]=\sum_{\beta\in\Dp}\chi([u,\ebeta])\varphi^{\beta},\\
[{\dst}_{\la}\varphi^{\alpha}]=-\frac{1}{2}\sum_{\beta,\gamma\in\Dp}(-1)^{\palpha\pbeta}c_{\beta,\gamma}^{\alpha}:\varphi^{\beta}\varphi^{\gamma}:,\quad
&[{\dne}_{\la}\bPhi_{\alpha}]=\sum_{\beta\in\Dn}\chi([\ebeta,\ealpha])\varphi^{\beta},\\
[{\dne}_{\la}\varphi^{\alpha}]=[{\dch}_{\la}\varphi^{\alpha}]=[{\dst}_{\la}\bPhi_{\alpha}]=[{\dch}_{\la}\bPhi_{\alpha}]=0,\hspace{5mm}
\end{align*}
where
\begin{align*}
&\bar{\tau}_{\p}(u|v)=(u|v)+\frac{\p}{2}\{\kappa_{\g}(u|v)-\kappa_{\g_{0}}(u|v)-2h^{\vee}(u|v)\},\\
&\bar{\mathit{a}}_{\p}(u|v)=(u|v)+\p\{\str_{\g}((\ad u)p_{+}(\ad v))-h^{\vee}(u|v)\},\\
&\bar{\mathit{b}}_{\p}(u|v)=(u|v)+\p\{\str_{\g}(p_{+}(\ad u)(\ad v))-h^{\vee}(u|v)\}.
\end{align*}
\end{lemma}

Replace $\p$ in the $\la$-brackets in Lemma \ref{qcal lemma} by an independent variable $\X$ and denote by $\hC_{\X}$ the vertex superalgebra over $\C[\X]$ generated by $\bJ^{u}(z)$ for all $u\in\g_{\leq0}$, $\varphi^{\alpha}(z)$ for all $\alpha\in\Dp$ and $\bPhi_{\alpha}(z)$ for all $\alpha\in\Dn$. The vertex superalgebra $\hC_{\X}$ over $\C[\X]$ is a cochain complex over $\C[\X]$ with the differential $d_{(0)}$ by Lemma \ref{qcal lemma}. Define a vertex superalgebra $\hC_{\p}$ by
\begin{align*}
\hC_{\p}=\hC_{\X}\underset{\C[\X]}{\otimes}\C_{\p}
\end{align*}
for $\p\in\C$, where $\C_{\p}$ is a one-dimensional $\C[\X]$-module defined by $\X=\p$ in $\C_{\p}$. The vertex superalgebra $\hC_{\p}$ is also a complex with the differential $d_{(0)}$. By construction,
\begin{align*}
\hC_{\p}\simeq C_{k}
\end{align*}
if $\p=(k+h^{\vee})^{-1}$.

Let $\hpC_{\X}$ be a vertex subalgebra over $\C[\X]$ in $\hC_{\X}$ generated by $\bJ^{u}(z)$ for all $u\in\g_{\leq0}$ and $\varphi^{\alpha}(z)$ for all $\alpha\in\Dp$. The vertex superalgebra $\hpC_{\X}$ over $\C[\X]$ is a chain complex over $\C[\X]$ with the differential ${\dst}_{(0)}$ by Lemma \ref{qcal lemma} and the formula ${\dst}_{(0)}^{2}=0$. Denote $\hpC_{\p}=\hpC_{\X}\otimes_{\C[\X]}\C_{\p}$. Thus,
\begin{align*}
\hpC_{\p}\otimes\Fne=\hC_{\p}
\end{align*}
and $\hpC_{\p}$ is also a complex with the differential ${\dst}_{(0)}$. Let
\begin{align*}
C_{\infty}=\hC_{\X}|_{\X=0},\quad\pC_{\infty}=\hpC_{\X}|_{\X=0}.
\end{align*}
Then $C_{\infty}$ is a complex with the differential $d_{(0)}$. It has a supercommutative vertex superalgebra structure by Lemma \ref{qcal lemma} in the case of $\p=0$ and analogously for $\pC_{\infty}$ and ${\dst}_{(0)}$. The vertex superalgebra $C_{\infty}$ is a Poisson vertex superalgebra through the Poisson $\la$-bracket
\begin{align*}
\{A_{\la}B\}=\frac{1}{\X}[A_{\la}B]_{\X=0},
\end{align*}
for $A,B\in C_{\infty}$, where the $\la$-bracket is that of $\hC_{\X}$. Then $C_{\infty}$ is called the {\it classical limit} of $C_{k}$ and $H(C_{\infty},d_{(0)})$ is called the {\it classical $\W$-algebra associated with $\g$, $f$, $k$ and $\Gamma$}. The cohomology $H(C_{\infty},d_{(0)})$ inherits the Poisson vertex superalgebra structure from $C_{\infty}$.

\section{Main Results}

\subsection{Restricted root system and screening operators}\label{Screening sec}
In this section we fix some notations that will be used in the sequel. Let
\begin{align*}
\Del=\Delta\backslash\Delta_{0},
\end{align*}
which is a restricted root system (\cite{BG}),
\begin{align}
\label{Pp eq}
\Pp=\{\alpha\in\Dp\mid\alpha\ \mathrm{is}\ \mathrm{indecomposable}\ \mathrm{in}\ \Dp\}
\end{align}
be the base of $\Del$ such that the positive part of $\Del$ coincides with a set of positive roots in $\Del$ and $Q_{0}$ be the root lattice of $\g_{0}$ i.e. $Q_{0}=\bigoplus_{\alpha\in\Pi_{0}}\Z\alpha$. Set $\Pp_{j}=\Pp\cap\Delta_{j}$. We define the equivalence relation in $\Pp$ by
\begin{align}
\label{Pprel eq}
\alpha\sim\beta\iff\alpha-\beta\in Q_{0}
\end{align}
for $\alpha,\beta\in\Pp$. Denote by $[\Pp]$ the quotient set of $\Pp$ and by $[\alpha]$ the equivalent class of $\alpha\in\Pp$ in $[\Pp]$.

\begin{lemma}
\label{Pp dec lemma}
$\Pp=\Pp_{\frac{1}{2}}\sqcup\Pp_{1}$.
\end{lemma}
\begin{proof}
If there exists $\alpha\in\Pp_{j}$ for $j>1$, the positive root $\alpha$ is not simple because
\begin{align*}
\Pi=\Pi_{0}\sqcup\Pi_{\frac{1}{2}}\sqcup\Pi_{1}
\end{align*}
as is shown in \cite{EK, Ho}. Thus, $\alpha$ is decomposed to two positive roots in $\Dp$, which contradicts that $\alpha\in\Pp$.
\end{proof}

Lemma \ref{Pp dec lemma} induces the decomposition
\begin{align*}
[\Pp]=[\Pp_{\frac{1}{2}}]\sqcup[\Pp_{1}].
\end{align*}

\begin{lemma}
Let
\begin{align*}
\C^{[\beta]}=\bigoplus_{\alpha\in[\beta]}\C x_{\alpha}
\end{align*}
be the $\g_{0}$-module defined by
\begin{align*}
u\cdot x_{\alpha}=\sum_{\gamma\in[\beta]}c_{\gamma,u}^{\alpha}x_{\gamma}
\end{align*}
for $u\in\g_{0}$, where the complex number $c_{\gamma,u}^{\alpha}$ is defined by $[e_{\gamma},u]=\sum_{\alpha\in[\beta]}c_{\gamma,u}^{\alpha}e_{\alpha}$. Then $\C^{[\beta]}$ is well-defined.
\end{lemma}
\begin{proof}
We show that
\begin{align*}
u\cdot v\cdot x_{\alpha}-(-1)^{p(u)p(v)}v\cdot u\cdot x_{\alpha}=[u,v]\cdot x_{\alpha}
\end{align*}
for all $u,v\in\g_{0}$ and $\alpha\in[\beta]$. Thus, it is enough to show that
\begin{align}
\label{Cbeta eq}
\sum_{\rho\in[\beta]}(c_{\gamma,u}^{\rho}c_{\rho,v}^{\alpha}-(-1)^{p(u)p(v)}c_{\gamma,v}^{\rho}c_{\rho,u}^{\alpha})=c_{\gamma,[u,v]}^{\alpha}
\end{align}
for all $\gamma\in[\beta]$. By Jacobi identity,
\begin{align*}
[e_{\gamma},[u,v]]=[[e_{\gamma},u],v]-(-1)^{p(u)p(v)}[[e_{\gamma},v],u].
\end{align*}
This implies \eqref{Cbeta eq}. Therefore the assertion follows.
\end{proof}

\begin{rem}
Let $H(\g_{+},\C)$ be the Lie superalgebra cohomology with the coefficients in the trivial $\g_{+}$-module $\C$. Then $H^{1}(\g_{+},\C)$ is naturally a $\g_{0}$-module induced by the adjoint action of $\g_{0}$ on $\g_{+}$, and
\begin{align*}
H^{1}(\g_{+},\C)\simeq\bigoplus_{[\beta]\in[\Pp]}\C^{[\beta]}
\end{align*}
is the irreducilble $\g_{0}$-module decomposition described in \cite{Kos}.
\end{rem}

Denote by $\hat{\g}_{0}=\g_{0}[t,t^{-1}]\oplus\C K$ the central extension of the loop superalgebra $\g_{0}[t,t^{-1}]$ via the even supersymmetric invariant bilinear form $\tau_{k}$ defined by \eqref{tauk eq}. We extend the $\g_{0}$-module $\C^{[\beta]}$ to a $\g_{0}[t]\oplus\C K$-module by $\g_{0}[t]t=0$ and $K=1$ on $\C^{[\beta]}$. Define a $\hat{\g}_{0}$-module $M_{[\beta]}$ by
\begin{align}
\label{M eq}M_{[\beta]}&=\mathrm{Ind}^{\hat{\g}_{0}}_{\g_{0}[t]\oplus\C K}\C^{[\beta]}
=U(\hat{\g}_{0})\underset{U(\g_{0}[t]\oplus\C K)}{\otimes}\C^{[\beta]}\\
&\simeq\Vtau\otimes\bigoplus_{\alpha\in[\beta]}\C x_{\alpha}.
\end{align}
Then $M_{[\beta]}$ is naturally a $\Vtau$-module. 

\begin{lemma}
\label{Lz lemma}
Suppose that $k+h^{\vee}\neq0$. Let $L(z)$ be the field on $\Vtau$ defined by
\begin{align*}
L(z)=\sum_{n\in\Z}L_{n}z^{-n-2}=\frac{1}{2(k+h^{\vee})}\sum_{i=1}^{\dim\g_{0}}:J^{u^{i}}(z)J^{u_{i}}(z):,
\end{align*}
where $\{u_{i}\}_{i=1}^{\dim\g_{0}}$ is a basis of $\g_{0}$ and $\{u^{i}\}_{i=1}^{\dim\g_{0}}$ is the dual basis of $\g_{0}$ such that $(u^{i}|u_{j})=\delta_{i,j}$. Then $L(z)$ is a Virasoro field on $\Vtau$.
\end{lemma}
\begin{proof}
Consider the decomposition
\begin{align*}
\g_{0}=\bigoplus_{s=0}^{n}\g'_{s},
\end{align*}
where $\g'_{0}$ is the center of $\g_{0}$ and $\g'_{s}$ is a simple Lie superalgebras for $s>0$. Let $r_{s}$ for $s=1,\ldots,n$ be the complex number defined by
\begin{align*}
\kappa_{\g_{0}}(u|v)=2r_{s}(u|v)
\end{align*}
for $u,v\in\g'_{s}$. Set $r_{0}=0$. Then, by definition,
\begin{align*}
\tau_{k}(u|v)=(k+h^{\vee}-r_{s})(u|v)
\end{align*}
for $u,v\in\g'_{s}$ and $s=0,\ldots,n$. Let $\{u_{s,i}\}_{i=1}^{\dim\g'_{s}}$ be a basis of $\g'_{s}$ and $\{u_{s}^{i}\}_{i=1}^{\dim\g'_{s}}$ be the dual basis of $\g'_{s}$ so that $(u_{s}^{i}|u_{s,j})=\delta_{i,j}$. By the Sugawara construction, we have a Virasoro field on $\Vtau$ defined by
\begin{align*}
L(z)=\frac{1}{2(k+h^{\vee})}\sum_{s=0}^{n}\sum_{i=1}^{\dim\g'_{s}}:J^{u_{s}^{i}}(z)J^{u_{s,i}}(z):.
\end{align*}
Therefore the assertion follows.
\end{proof}

\begin{rem}\label{Vtau rem}
There exists the canonical isomorphism
\begin{align*}
\Vtau=\bigotimes_{s=0}^{n}V^{k+h^{\vee}-r_{s}}(\g'_{s}),
\end{align*}
where we follow the notation in the proof of Lemma \ref{Lz lemma}.
\end{rem}

\begin{lemma}
\label{Lxalpha lemma}
\begin{align*}
L_{-1}x_{\alpha}=-\frac{(k+h^{\vee})^{-1}}{(e_{\alpha}|e_{-\alpha})}\sum_{\begin{subarray}{c}\beta\in[\alpha]\\ \gamma\in I\cup\Delta_{0}\end{subarray}}(-1)^{\pbeta\pgamma}c_{\beta,-\alpha}^{\gamma}J^{e_{\gamma}}_{(-1)}x_{\beta}.
\end{align*}
\end{lemma}
\begin{proof}
Direct calculations.
\end{proof}

We introduce a formal power series
\begin{align*}
S^{\alpha}(z)=\sum_{n\in\Z}S_{n}^{\alpha}z^{-n}
\end{align*}
for $\alpha\in\Pp$, where its coefficients are the operators
\begin{align*}
S^{\alpha}_{n}:\Vtau\rightarrow M_{[\alpha]}
\end{align*}
defined by
\begin{align}
\label{Screening def}
S^{\alpha}(z)A:=(-1)^{\palpha p(A)+p(A)}\ \e^{z L_{-1}}Y(A,-z)\ x_{\alpha}
\end{align}
for all $A\in\Vtau$, where $Y(A,z)$ is the vertex operator of $A$ on $M_{[\alpha]}$. The formal power series $S^{\alpha}(z)$ satisfies the following formulas for generic $k$ (see Definition \ref{generic def}).

\begin{prop}
\label{Screening prop}
If $k$ is generic,
\begin{align*}
&S^{\alpha}_{0}\ket{0}=x_{\alpha},\quad
S^{\alpha}_{n}\ket{0}=0\quad\foral\ n\geq1,\\
&[J^{u}_{(m)},S^{\alpha}_{n}]=\sum_{\beta\in[\alpha]}c_{\beta,u}^{\alpha}S^{\beta}_{m+n},\\
&\der S^{\alpha}(z)=-\frac{(k+h^{\vee})^{-1}}{(e_{\alpha}|e_{-\alpha})}\sum_{\begin{subarray}{c}\beta\in[\alpha]\\ \gamma\in I\cup\Delta_{0}\end{subarray}}(-1)^{\pbeta\pgamma}c_{\beta,-\alpha}^{\gamma}:J^{e_{\gamma}}(z)S^{\beta}(z):
\end{align*}
for all $u\in\g_{0}$ and $\alpha\in\Pp$.
\end{prop}

Proposition \ref{Screening prop} will be proved in Section \ref{VAstr sec}.

We define the {\it screening operators}
\begin{align*}
Q_{[\beta]}:\Vtau\otimes\Fne\rightarrow M_{[\beta]}\otimes\Fne
\end{align*}
for $[\beta]\in[\Pp]$ by
\begin{align}
\label{Screening eq1}&Q_{[\beta]}=\sum_{\alpha\in[\beta]}\int:S^{\alpha}(z)\Phi_{\alpha}(z):dz\quad\mathrm{if}\ [\beta]\in[\Pp_{\frac{1}{2}}],\\
\label{Screening eq2}&Q_{[\beta]}=\sum_{\alpha\in[\beta]}\chi(e_{\alpha})\int S^{\alpha}(z)dz\hspace{8.5mm}\mathrm{if}\ [\beta]\in[\Pp_{1}],
\end{align}
where we have set
\begin{align*}
\int a(z)\ dz=a_{(0)}
\end{align*}
for a formal power series $a(z)=\sum_{n\in\Z}a_{(n)}z^{-n-1}$.

\subsection{Main Theorems}\label{main thm sec}
\begin{theorem}
\label{main1 thm}
If $k$ is generic, $\W^{k}(\g,f;\Gamma)$ may be described as a vertex subalgebra in $\Vtau\otimes\Fne$, which is the intersection of kernels of the screening operators:
\begin{align}\label{main1 eq}
 \W^{k}(\g,f;\Gamma)\simeq\bigcap_{[\beta]\in[\Pp]}\Ker\ Q_{[\beta]}.
\end{align}
Here the screening operators
\begin{align*}
Q_{[\beta]}:\Vtau\otimes\Fne\rightarrow M_{[\beta]}\otimes\Fne
\end{align*}
for $[\beta]\in[\Pp]$ are defined by \eqref{Screening eq1} and \eqref{Screening eq2}.
\end{theorem}

 The proof of Theorem \ref{main1 thm} is given in Section \ref{proof sec}.

 Suppose that $\g_{0}=\h$. Let $\Hi$ be the Heisenberg vertex algebra associated with $\h^{*}$ and its non-degenerate bilinear form $\inv$ induced by $\h$, $\Hi_{\beta}$ be the $\Hi$-module with the highest weight $\beta\in\h^{*}$ and $\e^{\int\beta(z)}$ be the formal power series whose coefficients are operators from $\Hi$ to $\Hi_{\beta}$ defined by
\begin{align*}
\e^{\int\beta(z)}=s_{\beta}z^{\beta_{(0)}}\exp(-\sum_{n<0}\frac{\beta_{(n)}}{n}z^{-n})\exp(-\sum_{n>0}\frac{\beta_{(n)}}{n}z^{-n}),
\end{align*}
where $s_{\beta}$ is the shift operator determined by the following formulas
\begin{align*}
s_{\beta}\ket{0}=\ket{\beta},\quad
[s_{\beta},\alpha_{(n)}]=0\quad\foral\ n\neq0
\end{align*}
for $\alpha\in\h^{*}$, where $\ket{\beta}$ is the highest weight vector of $\Hi_{\beta}$.

\begin{theorem}
\label{main2 thm}
If $k$ is generic and $\g_{0}=\h$, then $\W^{k}(\g,f)$ may be described as a vertex subalgebra in $\Hi\otimes\Fne$ as follows:
\begin{align*}
\W^{k}(\g,f)\simeq
\bigcap_{\alpha\in\Pi_{\frac{1}{2}}}\Ker\int:\e^{-\frac{1}{\kh}\int\alpha(z)}\Phi_{\alpha}(z):dz\cap\bigcap_{\begin{subarray}{c}\alpha\in\Pi_{1}\\ \chi(e_{\alpha})\neq 0\end{subarray}}\Ker\int \e^{-\frac{1}{\kh}\int\alpha(z)}\ dz,
\end{align*}
where $\kh=\sqrt{\mathstrut k+h^{\vee}}$.
\end{theorem}

\begin{proof}
If $k$ is generic and $\g_{0}=\h$, then $\Dp=\Delta_{+}$ and $\Pp=\Pi$. Since $Q_{0}$ is empty, the equivalence relation in $\Pp$ defined by \eqref{Pprel eq} is trivial. Therefore
\begin{align*}
\W^{k}(\g,f)
\simeq&\bigcap_{\alpha\in\Pi_{\frac{1}{2}}}\Ker\int:S^{\alpha}(z)\Phi_{\alpha}(z):dz\cap\bigcap_{\alpha\in\Pi_{1}}\Ker\ \chi(\ealpha)\int S^{\alpha}(z)\ dz\\
\simeq&\bigcap_{\alpha\in\Pi_{\frac{1}{2}}}\Ker\int:S^{\alpha}(z)\Phi_{\alpha}(z):dz\cap\bigcap_{\begin{subarray}{c}\alpha\in\Pi_{1}\\ \chi(e_{\alpha})\neq 0\end{subarray}}\Ker\int S^{\alpha}(z)\ dz
\end{align*}
by Theorem \ref{main1 thm}. The vertex algebra $\Vtau=V^{\tau_{k}}(\h)$ is generated by $J^{u}(z)$ for all $u\in\h$ satisfying
\begin{align*}
[{J^{u}}_{\la}J^{v}]=(k+h^{\vee})(u|v)\la
\end{align*}
for $u,v\in\h$ since
\begin{align*}
\tau_{k}(u|v)
&=k(u|v)+\frac{1}{2}\kappa_{\g}(u|v)-\frac{1}{2}\kappa_{\g_{0}}(u|v)\\
&=k(u|v)+\frac{1}{2}\kappa_{\g}(u|v)\\
&=(k+h^{\vee})(u|v).
\end{align*}
We identify $\h$ with $\h^{*}$ via the non-degenerate bilinear form $\inv$, that is, $\h^{*}\ni \beta\ \mapsto\ t_{\beta}\in\h$ such that $(t_{\beta}|h)=\beta(h)$ for all $h\in\h$. Denote
\begin{align*}
\beta(z)=\frac{1}{\kh}J^{t_{\beta}}(z)
\end{align*}
for $\beta\in\h^{*}$, where $\kh=\sqrt{\mathstrut k+h^{\vee}}$. Then
\begin{align*}
[\beta_{\la}\gamma]=(\beta|\gamma)\la
\end{align*}
for $\beta,\gamma\in\h^{*}$. Hence $V^{\tau_{k}}(\h)$ is the Heisenberg vertex algebra $\Hi$ associated with $\h^{*}$. We show
\begin{align}
\label{free eq}
S^{\alpha}(z)=\e^{-\frac{1}{\kh}\int\alpha(z)}
\end{align}
for $\alpha\in\Pi$. We have
\begin{align*}
[\ealpha,e_{-\alpha}]=(\ealpha|e_{-\alpha})t_{\alpha}\quad\mathrm{for}\ \mathrm{all}\ \alpha\in\Delta_{+}.
\end{align*}
by the invariant properties of $\inv$. Using this formula and Proposition \ref{Screening prop}, we obtain the following relations:
\begin{align*}
&S^{\alpha}_{0}\ket{0}=x_{\alpha}\ \mathrm{and}\ S^{\alpha}_{n}\ket{0}=0\ \mathrm{for}\ \mathrm{all}\ n\geq1,\\
&[\beta_{(m)},S^{\alpha}_{n}]=-\frac{1}{\kh}(\beta|\alpha)S^{\alpha}_{m+n},\\
&\der S^{\alpha}(z)=-\frac{1}{\kh}:\alpha(z)S^{\alpha}(z):
\end{align*}
for $\alpha\in\Pi$ and $\beta\in\h^{*}$. This implies \eqref{free eq}. Therefore the assertion follows.
\end{proof}

\begin{rem}
For $\g=\slf_{n}$, there is a one-to-one correspondence between nilpotent orbits in $\slf_{n}$ and partitions of $n$. If $f$ is a nilpotent element corresponding to the partition $(s,t)$ with $s+t=n,s>t$, there exists a good grading for $f$ such that  $(\slf_{n})_{0}=\h$ (see \cite{EK}).
\end{rem}

\section{Proof of Theorem \ref{main1 thm}}\label{proof sec}

\subsection{Weight filtration}\label{weight subsec}
We define a degree map $\deg:C_{k}\rightarrow\Z$ by
\begin{align*}
&\deg J^{u}=-2j\hspace{4.5mm}\foral\ u\in\g_{j}\ (j\leq0),\\
&\deg\varphi^{\alpha}=2j\hspace{7mm}\foral\ \alpha\in\Delta_{j}\ (j>0),\\
&\deg\Phi_{\alpha}=0\hspace{8.5mm}\foral\ \alpha\in\Dn,\\
&\deg\ket{0}=0,\quad
\deg\der A=\deg A\hspace{3mm}\foral\ A\in C_{k},\\
&\deg:AB:=\deg A+\deg B\hspace{7.5mm}\foral\ A,B\in C_{k}.
\end{align*}
Let
\begin{align*}
F_{p}C_{k}=\{A\in C_{k}\mid\deg A\geq p\}
\end{align*}
and $F_{\bullet}C_{k}=\{F_{p}C_{k}\}_{p\geq0}$. Since
\begin{align*}
d_{(0)}(F_{p}C_{k})\subset F_{p}C_{k}
\end{align*}
for all $p\geq0$, $F_{\bullet}C_{k}$ is a filtration of a complex $C_{k}$. Thus, the filtration $F_{\bullet}C_{k}$ gives a spectral sequence $\{E_{n}\}_{n\geq0}$. We call $F_{\bullet}C_{k}$ a {\it weight filtration} on $C_{k}$. Define the conformal weight of $C_{k}$ inherited from $C^{k}(\g,f;\Gamma)$, that is,
\begin{align*}
&\conf(J^{u})=1-j\quad\foral\ u\in\g_{j}\ (j\leq0),\\
&\conf(\varphi^{\alpha})=j\hspace{9.5mm}\foral\ \alpha\in\Delta_{j}\ (j>0),\\
&\conf(\Phi_{\alpha})=\frac{1}{2}\hspace{8.5mm}\foral\ \alpha\in\Dn,\\
&\conf(\ket{0})=0,\quad\conf(\der A)=\conf(A)+1\quad\foral\ A\in C_{k},\\
&\conf(:AB:)=\conf(A)+\conf(B)\hspace{10mm}\quad\foral\ A,B\in C_{k},
\end{align*}
where we denote by $\conf(A)\in\frac{1}{2}\Z$ the conformal weight of $A\in C_{k}$. Let $C_{k}(n)$ be the homogeneous conformal weight space with the conformal weight $n$ in $C_{k}$. We have
\begin{align}
\label{conf dec eq}
C_{k}=\bigoplus_{n\geq0}C_{k}(n).
\end{align}
Since $d_{(0)}$ preserves each homogeneous conformal weight spaces, $C_{k}(n)$ is a subcomplex of $C_{k}$. Hence the filtration $F_{\bullet}C_{k}$ induces a filtration and a spectral sequence on the subcomplex $C_{k}(n)$. Since $\dim C_{k}(n)<\infty$, the spectral sequence on $C_{k}$(n) converges for all $n\geq0$ so that the total spectral sequence $\{E_{n}\}_{n\geq0}$ also converges. Denote by $d_{r}$ the differential of $E_{r}$ and
\begin{align*}
d_{r}^{(i)}:E_{r}^{(i)}\rightarrow E_{r}^{(i+1)}\quad \mathrm{for}\ i\geq0.
\end{align*}
By construction,
\begin{align*}
d_{0}={\dst}_{(0)},\quad
d_{1}={\dne}_{(0)},\quad
d_{2}={\dch}_{(0)}
\end{align*}
and $d_{r}=0$ for $r\geq3$.

\begin{lemma}
\label{E1 lemma}
$E_{1}=H(\pC_{k},\dstz)\otimes\Fne$.
\end{lemma}
\begin{proof}
The assertion follows from the fact that the differential ${\dst}_{(0)}$ acts only on the subcomplex $\pC_{k}$ and \eqref{pCk eq}.
\end{proof}

\subsection{The first spectral sequence}\label{E1 subsec}
Suppose that $V$ is a super space. Then we denote by $S(V)$ the supersymmetric algebra of $V$, which is the tensor algebra of $V$ quotient by the two-sided ideal generated by $a b-(-1)^{\bar{a}\bar{b}}b a$ for all $a,b\in V$.

Let $\JS$ be the supersymmetric algebra of $\g_{0}[t^{-1}]t^{-1}$ and $H(\g_{+},\C)$ be the Lie superalgebra cohomology with the coefficients in the trivial $\g_{+}$-module $\C$. The Chevalley complex of $H(\g_{+},\C)$ is the superexterior algebra $\Lambda(\g_{+})$. We denote by $J^{u}(-n)$ the element of $\JS$ corresponding to $u t^{-n}\in\g_{0}[t^{-1}]t^{-1}$ and by $\varphi^{\alpha}(0)$ the element of $\Lambda(\g_{+})$ corresponding to $e_{\alpha}\in\g_{+}$ with reversed parity for $\alpha\in\Dp$. We can regard $\JS$ and $\Lambda(\g_{+})$ as subalgebras of the supercommutative algebra $\pC_{\infty}$ by
\begin{align*}
\JS\ni\ J^{u_{1}}(-n_{1})\cdots J^{u_{s}}(-n_{s})\ &\longmapsto\ \bJ^{u_{1}}_{(-n_{1})}\cdots\bJ^{u_{s}}_{(-n_{s})}\ket{0}\ \in\pC_{\infty},\\
\Lambda(\g_{+})\ni\hspace{8mm}\varphi^{\beta_{1}}(0)\cdots\varphi^{\beta_{s}}(0)\quad&\longmapsto\quad\varphi^{\beta_{1}}_{0}\cdots\varphi^{\beta_{s}}_{0}\ket{0}\hspace{7.5mm}\in\pC_{\infty}
\end{align*}
for $u_{i}\in\g_{0}$ and $\beta_{i}\in\Dp$.

\begin{lemma}
\label{Cinf lemma}
$H(\pC_{\infty},\dstz)\simeq\JS\otimes H(\g_{+},\C)$.
\end{lemma}

The proof of Lemma \ref{Cinf lemma} will be given in Section \ref{Cinf subsec}.

\begin{cor}
\label{CX cor}
$H(\hpC_{\X},\dstz)\simeq H(\pC_{\infty},\dstz)\otimes\C[\X]$.
\end{cor}
\begin{proof}
Let 
\begin{align*}
F_{p}\hpC_{\X}=\X^{p}\hpC_{\X}
\end{align*}
for all $p\geq0$. Since
\begin{align*}
\dstz(F_{p}\hpC_{\X})\subset F_{p}\hpC_{\X}
\end{align*}
for all $p\geq0$, $\{F_{p}\hpC_{\X}\}_{p\geq0}$ is a filtration of the complex $\hpC_{\X}$ over $\C[\X]$ and gives a spectral sequence $\{E_{\X,n}\}_{n\geq0}$. We extend the conformal weight on $\pC_{k}$ to $\hpC_{\X}$ by $\conf(\X)=0$ and get the homogeneous conformal weight decomposition of $\hpC_{\X}$. Since each of the homogeneous spaces is a free $\C[\X]$-module of finite rank, the induced spectral sequence on it converges so that the total spectral sequence $\{E_{\X,n}\}_{n\geq0}$ also converges. Denote by $d_{\X,r}$ the differential of $E_{\X,r}$ for $r\geq0$. It is straightforward that
\begin{align}
\label{EX1 eq}
E_{\X,1}=H(E_{\X,0},d_{\X,0})=H(\pC_{\infty})\otimes\C[\X].
\end{align}
By Lemma \ref{Cinf lemma}, 
\begin{align*}
H(\pC_{\infty})\simeq\JS\otimes H(\g_{+},\C)
\end{align*}
and this implies
\begin{align*}
E_{\X,1}=V^{\tau_{\X}}(\g_{0})\underset{\C[\X]}{\otimes}H(\g_{+},\C[\X]),
\end{align*}
where $V^{\tau_{\X}}(\g_{0})$ is a vertex subalgebra over $\C[\X]$ in $\hpC_{\X}$ generated by $\bJ^{u}(z)$ for all $u\in\g_{0}$ and $H(\g_{+},\C[\X])$ is the Lie superalgebra cohomology with the coefficients in the trivial $\g_{+}$-module $\C[\X]$ whose complex is identified with a subspace in $\hpC_{\X}$ spanned by all the elements
\begin{align*}
\varphi^{\beta_{1}}_{0}\cdots\varphi^{\beta_{s}}_{0}\ket{0}
\end{align*}
with $\beta_{i}\in\Dp$ and $s\geq0$ over $\C[\X]$. Since
\begin{align*}
\dstz V^{\tau_{\X}}(\g_{0})=\dstz H(\g_{+},\C[\X])=0,
\end{align*}
$d_{\X,r}=0$ for all $r\geq2$ and then the spectral sequence collapses at $E_{\X,1}=E_{\X,\infty}$. Therefore the assertion follows from \eqref{EX1 eq}.
\end{proof}

\begin{lemma}
The set
\begin{align*}
S=\{k\in\C\mid H(\pC_{k},\dstz)\simeq H(\pC_{\infty},\dstz)\}
\end{align*}
is Zariski dense in $\C$.
\end{lemma}
\begin{proof}
By Corollary \ref{CX cor},
\begin{align*}
H(\hpC_{\X},\dstz)\underset{\C[\X]}{\otimes}\C_{\p}\simeq H(\pC_{\infty},\dstz).
\end{align*}
Suppose that $k\in\C\backslash\{-h^{\vee}\}$. If
\begin{align}
\label{generic eq}
H(\hpC_{\X}\underset{\C[\X]}{\otimes}\C_{\p},\dstz)=H(\hpC_{\X},\dstz)\underset{\C[\X]}{\otimes}\C_{\p}
\end{align}
holds for $\p=(k+h^{\vee})^{-1}$, it follows that $k\in S$ because
\begin{align*}
\pC_{k}=\hpC_{\p}=\hpC_{\X}\underset{\C[\X]}{\otimes}\C_{\p}.
\end{align*}
Furthermore \eqref{generic eq} holds for $\p=0$ because
\begin{align*}
\pC_{\infty}=\hpC_{\X}\underset{\C[\X]}{\otimes}\C_{0}.
\end{align*}
Therefore
\begin{align}
\label{homo generic eq}
H(\hpC_{\X}(n)\underset{\C[\X]}{\otimes}\C_{\p},\dstz)=H(\hpC_{\X}(n),\dstz)\underset{\C[\X]}{\otimes}\C_{\p}
\end{align}
holds for $\p=0$ and all $n\in\frac{1}{2}\Z_{\geq0}$, where $\hpC_{\X}(n)$ is the homogeneous conformal weight space of conformal weight $n$ in $\hpC_{\X}$. Since $\hpC_{\X}(n)$ is a free $\C[\X]$-module of finite rank, the upper semi-continuous theorem (cf. \cite{Ha}) in a neighborhood of $\X=0$ implies that \eqref{homo generic eq} holds for some neighborhood $U(n)$ of $\p=0$ in $\C$. The complement $\C\backslash U(n)$ is a finite set because $U(n)$ is a non-empty Zariski open set in $\C$. Hence
\begin{align*}
\C\ \backslash\bigcap_{n\geq0}U(n)=\bigcup_{n\geq0}(\C\backslash U(n))
\end{align*}
is a countable set. Therefore a set
\begin{align*}
S'=\{k\in\C\backslash\{-h^{\vee}\}\mid(k+h^{\vee})^{-1}\in\bigcap_{n\geq0}U(n)\}
\end{align*}
is dense in $\C$. The assertion follows from the fact $S'\subset S$ because \eqref{generic eq} holds for $k\in S'$.
\end{proof}

\begin{definition}\label{generic def}
$k\in\C$ is called generic if $H(\pC_{k},\dstz)\simeq H(\pC_{\infty},\dstz)$.
\end{definition}

If $k$ is generic, Lemma \ref{Cinf lemma} implies that
\begin{align}
\label{E1 eq}
H(\pC_{k},\dstz)\simeq\Vtau\otimes H(\g_{+},\C),
\end{align}
where the complex $\Lambda(\g_{+})$ of $H(\g_{+},\C)$ is identified with the subspace of $\pC_{k}$ spanned by
\begin{align*}
\varphi^{\beta_{1}}_{0}\cdots\varphi^{\beta_{s}}_{0}\ket{0}
\end{align*}
for all $\beta_{i}\in\Dp$ and $s\geq0$.

\begin{prop}
\label{E1 prop}
If $k$ is generic,
\begin{align*}
E_{1}\simeq\Vtau\otimes H(\g_{+},\C)\otimes\Fne.
\end{align*}
\end{prop}
\begin{proof}
By Lemma \ref{E1 lemma},
\begin{align*}
E_{1}=H(\pC_{k},\dstz)\otimes\Fne.
\end{align*}
The assertion immediately follows from \eqref{E1 eq}.
\end{proof}

Recall that $\Pp$ is the set of all indecomposable roots in $\Dp$ defined in \eqref{Pp eq}.

\begin{lemma}
\label{dstPp lemma}
Let $\alpha\in\Pp$. Then
\begin{align*}
\dstz J^{e_{-\alpha}}=-\sum_{\begin{subarray}{c}\beta\in[\alpha]\\ \gamma\in I\cup\Delta_{0}\end{subarray}}(-1)^{\pgamma}c_{-\alpha,\beta}^{\gamma}:J^{\egamma}\varphi^{\beta}:+(e_{-\alpha}|e_{\alpha})(k+h^{\vee})\der\varphi^{\alpha}
\end{align*}
\end{lemma}
\begin{proof}
Let $\alpha\in\Pp$. We have
\begin{align*}
{\dst}_{(0)}J^{e_{-\alpha}}=-\sum_{\begin{subarray}{c}\beta\in\Dp\\ \gamma\in I\cup\Delta_{\leq0}\end{subarray}}(-1)^{\pgamma}c_{-\alpha,\beta}^{\gamma}:J^{\egamma}\varphi^{\beta}:+\mathit{a}_{k}(e_{-\alpha}|\ealpha)\der\varphi^{\alpha},
\end{align*}
where
\begin{align*}
\mathit{a}_{k}(e_{-\alpha}|\ealpha)=\str_{\g}((\ad e_{-\alpha})p_{+}(\ad \ealpha))+k(e_{-\alpha}|e_{\alpha}).
\end{align*}
The indecomposablility of $\alpha$ implies that $\beta\in[\alpha]$ if $c_{-\alpha,\beta}^{\gamma}\neq0$ for $\beta\in\Dp$ and $\gamma\in I\cup\Delta_{0}$. It is enough to show that
\begin{align}
\label{dstPp eq}
\mathit{a}_{k}(e_{-\alpha}|\ealpha)=(k+h^{\vee})(e_{-\alpha}|e_{\alpha}).
\end{align}
First,
\begin{align*}
\kappa_{\g}(e_{-\alpha}|\ealpha)=\str_{\g}((\ad e_{-\alpha})p_{+}(\ad \ealpha))+\str_{\g}((\ad e_{-\alpha})p_{\leq0}(\ad \ealpha)),
\end{align*}
where $p_{\leq0}:\g\rightarrow\g_{\leq0}$ is the projection map. Next, we have
\begin{align*}
\str_{\g}((\ad e_{-\alpha})p_{+}(\ad \ealpha))=\str_{\g}((\ad e_{-\alpha})p_{\leq0}(\ad \ealpha))
\end{align*}
by the indecomposablility of $\alpha$ and the invariant property of $\inv$. Thus,
\begin{align*}
\str_{\g}((\ad e_{-\alpha})p_{+}(\ad \ealpha))=\frac{1}{2}\kappa_{\g}(e_{-\alpha}|\ealpha)=h^{\vee}(e_{-\alpha}|\ealpha).
\end{align*}
Therefore the formula \eqref{dstPp eq} holds. This completes the proof.
\end{proof}

\subsection{The vertex superalgebra structure of $E_{1}$}\label{VAstr sec}
Notice that the cohomology $H(\pC_{k},\dstz)$ inherits a vertex superalgebra structure from $\pC_{k}$. Thus,
\begin{align*}
E_{1}=H(\pC_{k},\dstz)\otimes\Fne
\end{align*}
is a vertex superalgebra. Assume that $k$ is generic. By \eqref{E1 eq},
\begin{align*}
H(\pC_{k},\dstz)\simeq\Vtau\otimes H(\g_{+},\C).
\end{align*}
Recall that we denote by $\varphi^{\alpha}(0)$ the element of the complex $\Lambda(\g_{+})$ of $H(\g_{+},\C)$ corresponding to $e_{\alpha}\in\g_{0}$ with reversed parity for $\alpha\in\Dp$.

\begin{lemma}
\label{trivial lemma}
\begin{align*}
H^{0}(\g_{+},\C)=\C,\quad
H^{1}(\g_{+},\C)=\bigoplus_{\alpha\in\Pp}\C\psi_{\alpha},
\end{align*}
where $\psi_{\alpha}$ is the non-zero cohomology class of $\varphi^{\alpha}(0)$ in $H^{1}(\g_{+},\C)$.
\end{lemma}
\begin{proof}
Denote by $d_{+}$ be the differential of the complex $\Lambda(\g_{+})$ of $H(\g_{+},\C)$. Since
\begin{align*}
\Lambda^{0}(\g_{+})=\C,\quad
d_{+}\cdot1=0,
\end{align*}
we obtain the first formula. Moreover
\begin{align*}
\Lambda^{1}(\g_{+})=\bigoplus_{\alpha\in\Dp}\C\varphi^{\alpha}(0)
\end{align*}
and
\begin{align*}
d_{+}\cdot\varphi^{\alpha}(0)
&=-\frac{1}{2}\sum_{\beta,\gamma\in\Dp}(-1)^{\palpha\pbeta}c_{\beta,\gamma}^{\alpha}\varphi^{\beta}(0)\varphi^{\gamma}(0)\\
&=-\sum_{\begin{subarray}{c}\beta\prec\gamma\in\Dp\\ \alpha=\beta+\gamma\end{subarray}}(-1)^{\palpha\pbeta}c_{\beta,\gamma}^{\alpha}\varphi^{\beta}(0)\varphi^{\gamma}(0)-\frac{1}{2}\sum_{\begin{subarray}{c}\beta\in\Dp\\ \alpha=2\beta\end{subarray}}c_{\beta,\beta}^{\alpha}\varphi^{\beta}(0)^{2},
\end{align*}
where $\prec$ is any total order in $\Dp$. Therefore $d_{+}\cdot\varphi^{\alpha}(0)=0$ if and only if $\alpha$ is indecomposable in $\Dp$. This implies the second formula.
\end{proof}

Here,
\begin{align*}
&H^{0}(\pC_{k},\dstz)\simeq\Vtau,\\
&H^{1}(\pC_{k},\dstz)\simeq\Vtau\otimes\bigoplus_{\alpha\in\Pp}\C\psi_{\alpha}
\end{align*}
by Lemma \ref{trivial lemma}. The first formula is also a vertex superalgebra isomorphism. Let
\begin{align*}
\tS^{\alpha}(z)=\sum_{n\in\Z}\tS^{\alpha}_{n}z^{-n}=Y(\psi_{\alpha},z)
\end{align*}
be the vertex operator of $\psi_{\alpha}$ on $H(\pC_{k},\dstz)$ for $\alpha\in\Pp$.

\begin{lemma}
\label{Salpha lemma}
The following formulas hold:
\begin{align}
\label{Salpha eq1}&\tS^{\alpha}_{0}\ket{0}=\psi_{\alpha},\quad
\tS^{\alpha}_{n}\ket{0}=0\quad\mathrm{for}\ n\geq1,\\
\label{Salpha eq2}&[{J^{u}}_{\la}\tS^{\alpha}]=\sum_{\beta\in[\alpha]}c_{\beta,u}^{\alpha}\tS^{\beta},\\
\label{Salpha eq3}&\der\tS^{\alpha}(z)=-\frac{(k+h^{\vee})^{-1}}{(e_{\alpha}|e_{-\alpha})}\sum_{\begin{subarray}{c}\beta\in[\alpha]\\ \gamma\in I\cup\Delta_{0}\end{subarray}}(-1)^{\pbeta\pgamma}c_{\beta,-\alpha}^{\gamma}:J^{e_{\gamma}}(z)\tS^{\beta}(z):
\end{align}
for $\alpha\in\Pp$ and $u\in\g_{0}$.
\end{lemma}

\begin{proof}
The following formulas hold in $\pC_{k}$:
\begin{align*}
&\varphi^{\alpha}_{0}\ket{0}=\varphi_{\alpha}(0),\quad
\varphi^{\alpha}_{n}\ket{0}=0\quad\mathrm{for}\ n\geq1,\\
&[{J^{u}}_{\la}\varphi^{\alpha}]=\sum_{\beta\in[\alpha]}c_{\beta,u}^{\alpha}\varphi^{\beta}
\end{align*}
for $\alpha\in\Pp$ and $u\in\g_{0}$. These formulas imply \eqref{Salpha eq1} and \eqref{Salpha eq2}. Moreover
\begin{align*}
\dstz J^{e_{-\alpha}}\equiv0
\end{align*}
for all $\alpha\in\Pp$ in $H(\pC_{k},\dstz)$. Since
\begin{align*}
\dstz J^{e_{-\alpha}}=-\sum_{\begin{subarray}{c}\beta\in[\alpha]\\ \gamma\in I\cup\Delta_{0}\end{subarray}}(-1)^{\pgamma}c_{-\alpha,\beta}^{\gamma}:J^{\egamma}\varphi^{\beta}:+(e_{-\alpha}|\ealpha)(k+h^{\vee})\der\varphi^{\alpha}
\end{align*}
for all $\alpha\in\Pp$ by Lemma \ref{dstPp lemma}, the formula $\dstz J^{e_{-\alpha}}\equiv0$ holds if and only if
\begin{align*}
\der\varphi^{\alpha}(z)\equiv-\frac{(k+h^{\vee})^{-1}}{(e_{\alpha}|e_{-\alpha})}\sum_{\begin{subarray}{c}\beta\in[\alpha]\\ \gamma\in I\cup\Delta_{0}\end{subarray}}(-1)^{\pbeta\pgamma}c_{\beta,-\alpha}^{\gamma}:J^{e_{\gamma}}(z)\varphi^{\beta}(z):
\end{align*}
This implies \eqref{Salpha eq3} and completes the proof.
\end{proof}

By Lemma \ref{Salpha lemma}, $H^{1}(\pC_{k},\dstz)$ is a $\Vtau$-module and the operator $\tS^{\alpha}_{n}$ restricts to the maps
\begin{align*}
\tS^{\alpha}_{n}: \Vtau\rightarrow\Vtau\otimes\bigoplus_{\beta\in[\alpha]}\C\psi_{\beta}
\end{align*}
for all $\alpha\in\Pp$ and $n\in\Z$. Recall that $M_{[\alpha]}$ is a $\Vtau$-module defined by \eqref{M eq}.

\begin{lemma}
\label{Malpha lemma}
Let
\begin{align}
\label{Malpha eq}
H^{1}(\pC_{k},\dstz)\simeq\bigoplus_{[\alpha]\in[\Pp]}M_{[\alpha]}
\end{align}
be the vector space isomorphism defined by $A\psi_{\alpha}\mapsto A x_{\alpha}$ for all $A\in\Vtau$ and $\alpha\in\Pp$. Then this is also a $\Vtau$-module isomorphism.
\end{lemma}
\begin{proof}
By \eqref{Salpha eq2},
\begin{align*}
[J^{u}_{(n)},\tS^{\alpha}_{m}]=\sum_{\beta\in[\alpha]}c_{\beta,u}^{\alpha}\tS^{\beta}_{n+m}
\end{align*}
for $u\in\g_{0}$, $\alpha\in\Pp$, $m,n\in\Z$. Therefore
\begin{align*}
J^{u}_{(0)}\psi_{\alpha}=\sum_{\beta\in[\alpha]}c_{\beta,u}^{\alpha}\psi_{\beta},\quad
J^{u}_{(n)}\psi_{\alpha}=0\quad\mathrm{for}\ n\geq1
\end{align*}
by \eqref{Salpha eq1}. This implies that \eqref{Malpha eq} is also a $\Vtau$-module isomorphism.
\end{proof}

By Lemma \ref{Malpha lemma}, we can regard the operators $\tS^{\alpha}_{n}$ as the maps
\begin{align*}
\tS^{\alpha}_{n}: \Vtau\rightarrow M_{[\alpha]}.
\end{align*}
Recall that $S^{\alpha}(z)$ is the formal power series whose coefficients are the operators from $\Vtau$ to $M_{[\alpha]}$ defined in \eqref{Screening def}.

\begin{prop}\label{tS prop}
\begin{align*}
S^{\alpha}(z)=\tS^{\alpha}(z)
\end{align*}
as the operator from $\Vtau$ to $M_{[\alpha]}((z))$ for all $\alpha\in\Pp$.
\end{prop}
\begin{proof}
By the skew-symmetry in $H(\pC_{k},\dstz)$,
\begin{align*}
\tS^{\alpha}(z)A=(-1)^{\palpha p(A)+p(A)}\e^{z T}Y(A,-z)\psi_{\alpha}
\end{align*}
for all $A\in\Vtau$, where $Y(A,z)$ is the vertex operator on $H(\pC_{k},\dstz)$ and $T$ is the translation operator. All we need to show is that the actions of $L_{-1}$ and $T$ on $M_{[\alpha]}$ are the same via the isomorphism \eqref{Malpha eq}. However,
\begin{align}
\label{Sz eq1}L_{-1}(A x_{\alpha})&=(L_{-1}A)x_{\alpha}+A(L_{-1}x_{\alpha}),\\
\label{Sz eq2}T(A\psi_{\alpha})&=(T A)\psi_{\alpha}+A(T\psi_{\alpha})
\end{align}
for all $A\in\Vtau$ and
\begin{align*}
L_{-1}A=T A
\end{align*}
because $L(z)$ is a Virasoro field on $\Vtau$ by Lemma \ref{Lz lemma}. Therefore the first terms of \eqref{Sz eq1} and \eqref{Sz eq2} are the same. Thus, it is enough to show that
\begin{align*}
L_{-1}x_{\alpha}=T\psi_{\alpha}
\end{align*}
via the isomorphism \eqref{Malpha eq}. By \eqref{Salpha eq3},
\begin{align*}
T\psi_{\alpha}=-\frac{(k+h^{\vee})^{-1}}{(e_{\alpha}|e_{-\alpha})}\sum_{\begin{subarray}{c}\beta\in[\alpha]\\ \gamma\in I\cup\Delta_{0}\end{subarray}}(-1)^{\pbeta\pgamma}c_{\beta,-\alpha}^{\gamma}J^{e_{\gamma}}_{(-1)}\psi_{\beta}.
\end{align*}
This coincides with $L_{-1}x_{\alpha}$ by Lemma \ref{Lxalpha lemma}. So the proof is completed.
\end{proof}

\begin{proof}[Proof of Proposition \ref{Screening prop}]
By Proposition \ref{tS prop},
\begin{align*}
S^{\alpha}(z)=\tS^{\alpha}(z).
\end{align*}
Hence the assertion follows from Lemma \ref{Salpha lemma}.
\end{proof}

\subsection{Proof of Theorem \ref{main1 thm}}\label{proof subsec}
Assume that $k$ is generic. By Proposition \ref{E1 prop},
\begin{align*}
E_{1}=H(\pC_{k},\dstz)\otimes\Fne\simeq\Vtau\otimes H(\g_{+},\C)\otimes\Fne.
\end{align*}
Since $d_{1}$ is the vertex operator on $E_{1}$ induced by $\dnez$,
\begin{align*}
d_{1}=\sum_{\alpha\in\Pp_{\frac{1}{2}}}\int:\tS^{\alpha}(z)\Phi_{\alpha}(z):dz.
\end{align*}
Let $\td_{2}$ be a vertex operator on $E_{1}$ defined by
\begin{align*}
\td_{2}=\sum_{\alpha\in\Pp_{1}}\chi(e_{\alpha})\int\tS^{\alpha}(z)\ dz,
\end{align*}
which is the vertex operator induced by ${\dch}_{(0)}$.

\begin{lemma}\label{E1Q lem}
Assume that $k$ is generic. Let $Q$ be a vertex operator on $E_{1}$ defined by
\begin{align*}
Q=\sum_{\alpha\in\Pp_{\frac{1}{2}}}\int:\tS^{\alpha}(z)\Phi_{\alpha}(z):dz+\sum_{\alpha\in\Pp_{1}}\chi(e_{\alpha})\int\tS^{\alpha}(z)\ dz.
\end{align*}
Then $Q^{2}=0$ and $(E_{1},Q)$ is a complex with respect to the charge degrees. Moreover,
\begin{align*}
H(C_{k},d_{(0)})\simeq H(E_{1},Q).
\end{align*}
\end{lemma}
\begin{proof}
Recall that $Q=d_{1}+\td_{2}$. Since $[\dnez,\dchz]=[\dchz,\dchz]=0$,
\begin{align*}
Q^{2}=\frac{1}{2}[Q,Q]=\frac{1}{2}[d_{1}+\td_{2},d_{1}+\td_{2}]=\frac{1}{2}[d_{1},d_{1}]=d_{1}^{2}=0.
\end{align*}
Therefore $(E_{1},Q)$ is a complex with respect to the charge degrees. The filtration on $E_{1}$ induced by the weight filtration on $C_{k}$ gives a convergent spectral sequence $\{\tE_{n}\}$ such that
\begin{align*}
E_{n}\simeq \tE_{n}
\end{align*}
for all $n\geq1$. Therefore
\begin{align*}
H(C_{k},d_{(0)})\simeq E_{\infty}\simeq\tE_{\infty}\simeq H(E_{1},Q).
\end{align*}
This completes the proof.
\end{proof}

\begin{proof}[Proof of Theorem \ref{main1 thm}]
{\upshape Recall
\begin{align*}
\W^{k}(\g,f;\Gamma)=H(C_{k},d_{(0)})=H^{0}(C_{k},d_{(0)}).
\end{align*}
By Lemma \ref{E1Q lem},
\begin{align}\label{mainprf eq}
\W^{k}(\g,f;\Gamma)\simeq H^{0}(E_{1},Q)=\Ker\ Q.
\end{align}
Notice that
\begin{align*}
Q=\sum_{[\beta]\in[\Pp]}Q_{[\beta]}
\end{align*}
by definition in \eqref{Screening eq1} and \eqref{Screening eq2}, and that
\begin{align*}
Q_{[\beta]}\cdot(\Vtau\otimes\Fne)\subset M_{[\beta]}\otimes\Fne
\end{align*}
for $[\beta]\in[\Pp]$.  This implies the decomposition of the kernel of the operator $Q$:
\begin{align*}
\Ker\ Q=\bigcap_{[\beta]\in[\Pp]}\Ker\ Q_{[\beta]}.
\end{align*}
The isomorphism \eqref{mainprf eq} is the vertex superalgebra isomorphism (Lemma \ref{Miura lem} below). Therefore this completes the proof.
}
\end{proof}

\subsection{Proof of Lemma \ref{Cinf lemma}}\label{Cinf subsec}
Define the $\g_{+}[t]$ action on the affine superspace $Conn=\der+\g_{\geq0}[t]$ by
\begin{align*}
a(t)\cdot(\der+b(t))=\der+[a(t),b(t)]-\der a(t)
\end{align*}
for $a(t)\in\g_{+}[t]$ and $b(t)\in\g_{\geq0}[t]$. This action is induced by the derivation of the gauge action of the loop supergroup associated with $\g_{+}[t]$. This $\g_{+}[t]$ action on $Conn$ gives rise to the $\g_{+}[t]$-module structure on the space of the functions of $Conn$. Identify the space of the functions of $Conn$ with the supersymmetric algebra of $\g_{\leq0}[t^{-1}]t^{-1}$, denoted by $\JSn$, via the {\it residue dual}, that is, a non-degenerate bilinear form $\langle\hspace{0.5mm}\cdot\hspace{0.5mm}|\hspace{0.5mm}\cdot\hspace{0.5mm}\rangle_{\res}:Conn\times\g_{\leq0}\rightarrow\C$ defined by the following formula
\begin{align*}
\langle\der+a\otimes f(t)\mid b\otimes g(t)\rangle_{\res}=(a|b)\ \underset{t=0}{\mathrm{Res}}\ f(t)g(t)\ dt
\end{align*}
for $a\in\g_{\geq0}$, $b\in\g_{\leq0}$, $f(t)\in\C[t]$ and $g(t)\in\C[t^{-1}]t^{-1}$. Then $\JSn$ is a $\g_{+}[t]$-module. Set
\begin{align*}
u\otimes t^{n}=J^{u}(n)\in\JSn
\end{align*}
for $u\in\g_{\leq0}$, $n\in\Z_{<0}$ and
\begin{align*}
a\otimes t^{m}=a_{(m)}\in\g_{+}[t]
\end{align*}
for $a\in\g_{+}$, $m\in\Z_{\geq0}$. Here, $\g_{+}[t]$ acts on $\JSn$ by
\begin{align*}
{e_{\alpha}}_{(m)}\cdot J^{u}(-n)=-\sum_{\beta\in I\cup\Delta_{\leq0}}(-1)^{p(u)\palpha}c_{u,\alpha}^{\beta}J^{\ebeta}(-n+m)+n(e_{\alpha}|u)\delta_{m,n}
\end{align*}
for $\alpha\in\Dp$, $u\in\g_{\leq0}$, $m\geq0$ and $n>0$, and ${e_{\alpha}}_{(m)}$ acts on $\JSn$ as a super derivation with a parity $\palpha$. Consider the semi-infinite cohomology (cf. \cite{F,Fu}) of the $\g_{+}[t]$-module $\JSn$, denoted by
\begin{align}
\label{semi eq}
H(\g_{+}[t],\JSn).
\end{align}
The complex of \eqref{semi eq} is 
\begin{align*}
\JSn\otimes\Lambda((\g_{+}[t])^{*}),
\end{align*}
where $\Lambda((\g_{+}[t])^{*})$ is the superexterior algebra of $(\g_{+}[t])^{*}$. Let $\varphi^{\alpha}(-n)$ be a generator of $\Lambda((\g_{+}[t])^{*})$ associated with the element $(\ealpha\otimes t^{n})^{*}\in(\g_{+}[t])^{*}$ for $\alpha\in\Dp$ and $n\in\Z_{\geq0}$.

\begin{lemma}
$H(\g_{+},\JSn)=H(\pC_{\infty},\dstz)$.
\end{lemma}
\begin{proof}
The map $\pC_{\infty}\rightarrow\JSn\otimes\Lambda((\g_{+}[t])^{*})$ defined by
\begin{align*}
\ket{0}\mapsto 1,\quad
\bJ^{u}_{(-n)}\mapsto J^{u}(-n),\quad
\varphi^{\alpha}_{-m}\mapsto\varphi^{\alpha}(-m)
\end{align*}
for $u\in\g_{\leq0}$, $\alpha\in\Dp$, $n\geq1$, $m\geq0$ is a superalgebra isomorphism by construction. The derivation on $\pC_{\infty}$ coincides with the derivation on $\JSn\otimes\Lambda((\g_{+}[t])^{*})$ via this map by direct calculations. Therefore the assertion follows.
\end{proof}

\begin{proof}[Proof of Lemma \ref{Cinf lemma}]
Denote by
\begin{align*}
\Cinf=\JSn\otimes\Lambda((\g_{+}[t])^{*})
\end{align*}
the complex of \eqref{semi eq} and by $d_{\Cinf}$ the derivation of $\Cinf$. Let $F_{p}\Cinf$ be a subspace of $\Cinf$ spanned by all the elements
\begin{align*}
J^{u_{1}}(-n_{1})\cdots J^{u_{r}}(-n_{r})\otimes\varphi^{\beta_{1}}(-m_{1})\cdots\varphi^{\beta_{s}}(-m_{s})
\end{align*}
with $u_{i}\in\g_{\leq0}$, $\beta_{i}\in\Dp$, $n_{i}>0$, $m_{i}\geq0$ and $r+s\geq p$. Since 
\begin{align*}
d_{\Cinf}F_{p}\Cinf\subset F_{p}\Cinf,
\end{align*}
$F_{p}\Cinf$ is a filtration of the complex $\Cinf$ and gives the spectral sequence $\{E_{\Cinf,n}\}_{n\geq0}$ on $\Cinf$. The conformal weight on $\Cinf$ is defined in the same way as $\pC_{k}$.  Since each of homogeneous conformal weight spaces is finite dimensional, the spectral sequence restricted to each of homogeneous conformal weight spaces converges so that $\{E_{\Cinf,n}\}_{n\geq0}$ also converges. Set
\begin{align*}
E_{\Cinf,n+1}=H(E_{\Cinf,n},d_{\Cinf,n})
\end{align*}
for $n\geq0$. We have
\begin{align}
\label{Cinf eq1}&d_{\Cinf,0}\cdot J^{e_{-\alpha}}(-n)=n(e_{-\alpha}|\ealpha)\varphi^{\alpha}(-n),\\
\label{Cinf eq2}&d_{\Cinf,0}\cdot J^{u}(-n)=d_{\Cinf,0}\varphi^{\alpha}(-m)=0
\end{align}
for $\alpha\in\Dp$, $u\in\g_{0}$, $n>0$ and $m\geq0$. For $\alpha\in\Dp$ and $n\geq1$, denote by $\Cinf_{\alpha,n}$ a subalgebra of $\Cinf$ generated by $J^{e_{-\alpha}}(-n)$ and $\varphi^{\alpha}(-n)$. Let $\Cinf_{0}$ be a subalgebra generated by $\varphi^{\alpha}(0)$ for all $\alpha\in\Dp$. Then
\begin{align*}
\Cinf=\JS\otimes\Cinf_{0}\otimes\bigotimes_{\begin{subarray}{c}\alpha\in\Dp\\ n\geq1\end{subarray}}\Cinf_{\alpha,n}
\end{align*}
is a decomposition of the complex $\Cinf$ by \eqref{Cinf eq1} \eqref{Cinf eq2}. Therefore
\begin{align*}
H(\Cinf,d_{\Cinf,0})=\JS\otimes\Cinf_{0}\otimes\bigotimes_{\begin{subarray}{c}\alpha\in\Dp\\ n\geq1\end{subarray}}H(\Cinf_{\alpha,n},d_{\Cinf,0})
\end{align*}
since $d_{\Cinf,0}\cdot\JS=d_{\Cinf,0}\cdot\Cinf_{0}=0$. We show that it is sufficient to prove that
\begin{align}
\label{semi0 eq}
H(\Cinf_{\alpha,n},d_{\Cinf,0})=\C
\end{align}
for all $\alpha\in\Dp$ and $n\geq1$. Then
\begin{align*}
E_{\Cinf,1}=H(\Cinf,d_{\Cinf,0})=\JS\otimes\Cinf_{0}.
\end{align*}
Thus,
\begin{align*}
E_{\Cinf,2}=H(E_{\Cinf,1},d_{\Cinf,1})=\JS\otimes H(\Cinf_{0},d_{\Cinf,1})
\end{align*}
since $d_{\Cinf,1}\cdot\JS=0$. We have
\begin{align*}
d_{\Cinf,1}\cdot\varphi^{\alpha}(0)=-\frac{1}{2}\sum_{\beta,\gamma\in\Dp}(-1)^{\palpha\pbeta}c_{\alpha,\beta}^{\gamma}\varphi^{\beta}(0)\varphi^{\gamma}(0)
\end{align*}
for $\alpha\in\Dp$. This implies
\begin{align*}
H(\Cinf_{0},d_{\Cinf,1})=H(\g_{+},\C).
\end{align*}
Since $d_{\Cinf,r}=0$ for $r\geq2$, the assertion follows. Therefore we only need to show \eqref{semi0 eq} for all $\alpha\in\Dp$ and $n\geq1$. If $\palpha$ is even, $J^{e_{-\alpha}}(-n)$ is even and $\varphi^{\alpha}(-n)$ is odd. Therefore $\Cinf_{\alpha,n}$ is spanned by all the elements
\begin{align*}
(J^{e_{-\alpha}}(-n))^{r},\quad
(J^{e_{-\alpha}}(-n))^{r}\varphi^{\alpha}(-n)
\end{align*}
with $r\geq0$. We have
\begin{align*}
&d_{\Cinf,0}\cdot (J^{e_{-\alpha}}(-n))^{r}=r n(e_{-\alpha}|e_{\alpha})(J^{-\alpha}(-n))^{r-1}\varphi^{\alpha}(-n),\\
&d_{\Cinf,0}\cdot(J^{e_{-\alpha}}(-n))^{r}\varphi^{\alpha}(-n)=0
\end{align*}
by \eqref{Cinf eq1} and \eqref{Cinf eq2}. Hence \eqref{semi0 eq} follows. If $\palpha$ is odd, $J^{e_{-\alpha}}(-n)$ is odd and $\varphi^{\alpha}(-n)$ is even. Therefore $\Cinf_{\alpha,n}$ is spanned by all the elements
\begin{align*}
(\varphi^{\alpha}(-n))^{r},\quad
J^{e_{-\alpha}}(-n)(\varphi^{\alpha}(-n))^{r}
\end{align*}
with $r\geq0$. We have
\begin{align*}
&d_{\Cinf,0}\cdot J^{e_{-\alpha}}(-n)(\varphi^{\alpha}(-n))^{r}=n(e_{-\alpha}|e_{\alpha})(\varphi^{\alpha}(-n))^{r+1},\\
&d_{\Cinf,0}\cdot(\varphi^{\alpha}(-n))^{r}=0
\end{align*}
by \eqref{Cinf eq1} and \eqref{Cinf eq2}. Hence \eqref{semi0 eq} follows. Therefore this completes the proof.
\end{proof}

\section{Non-generic Level}\label{Miura sec}

In this section we consider the $\W$-algebra $\W^{k}(\g,f;\Gamma)$ at non-generic level. Recall that, following \cite{KW3}, there exists an ascending vertex superalgebra filtration on $C_{k}$ such that the induced filtration on $H(C_{k},d_{(0)})$ gives an isomorphism
\begin{align*}
\mathrm{gr}H(C_{k},d_{(0)})\simeq V^{\tau_{k}}(\g^{f}),
\end{align*}
where $\g^{f}\subset\g_{\leq0}$ is the centralizer of $f$ in $\g$. Hence if
\begin{align*}
\dim\g^{f}=n
\end{align*}
and $\{u_{1},\ldots,u_{n}\}$ is a basis of $\g^{f}$ such that
\begin{align*}
u_{i}\in\g_{-j_{i}}
\end{align*}
for all $i$, where all $j_{i}$'s are some non-negative half-integers, then the $\W$-algebra $\W^{k}(\g,f;\Gamma)$ is strongly generated by
\begin{align*}
W^{u_{i}}(z)=J^{u_{i}}(z)+(\mathrm{lower}\ \mathrm{terms})
\end{align*}
for all $i$ with conformal weights
\begin{align*}
\conf(W^{u_{i}})=j_{i}+1
\end{align*}
since the differential $d_{(0)}$ preserves the conformal weights on $C_{k}$. Thus, $\W^{k}(\g,f;\Gamma)$ is realized as a vertex subalgebra of $V^{\tau_{k}}(\g_{\leq0})\otimes \Fne$ generated by $W^{u_{i}}(z)$ for all $i$. Let
\begin{align*}
\tmu:V^{\tau_{k}}(\g_{\leq0})\otimes \Fne\rightarrow\Vtau\otimes\Fne
\end{align*}
be the vertex superalgebra homomorphism defined by $\tmu(J^{u}_{(n)})=0$ for all $u\in\g_{<0}$ and $n<0$, and $\tmu=\Id$ on $\Vtau\otimes\Fne$. This map $\widetilde{\mu}$ induces the vertex superalgebra homomorphism
\begin{align*}
\mu:\W^{k}(\g,f;\Gamma)\rightarrow V^{\tau_{k}}(\g_{0})\otimes \Fne
\end{align*}
by restriction, where we regard the $\W$-algebra $\W^{k}(\g,f;\Gamma)$ as a vertex subalgebra of $\Vtaun\otimes\Fne$. This map $\mu$ is called {\it the Miura map}.

\begin{lemma}\label{Miura lem}
The image of the isomorphism \eqref{main1 eq} in Theorem \ref{main1 thm} coincides with that of the Miura map $\mu$.
\end{lemma}
\begin{proof}
Assume that $k$ is generic. By Lemma \ref{E1Q lem},
\begin{align}\label{Miu eq}
\W^{k}(\g,f;\Gamma)=H^{0}(C_{k},d_{(0)})\simeq H^{0}(E_{1},Q),
\end{align}
which is the isomorphism \eqref{main1 eq}. This is the vector isomorphism induced by the weight filtration on $C_{k}$. By Proposition \ref{E1 prop},
\begin{align*}
E_{1}^{(0)}=\Vtau\otimes\Fne.
\end{align*}
This implies that the leading terms of the cohomology classes in $C_{k}^{(0)}$ concentrate on  the homogeneous space with degree $0$. Therefore the isomorphism \eqref{Miu eq} is the projection of the homogeneous terms with degree $0$ in the cohomology classes in $C_{k}^{(0)}$. Hence the image of \eqref{Miu eq} coincides with that of the Miura map by definition. This completes the proof.
\end{proof}

\begin{cor}\label{Miura cor}
If $k$ is generic, the Miura map $\mu$ is injective and there exist strongly generating fields $X^{1}(z),\ldots,X^{n}(z)$ of the vertex subalgebra
\begin{align*}
\mu(\W^{k}(\g,f;\Gamma))\simeq\W^{k}(\g,f;\Gamma)
\end{align*}
of $\Vtau\otimes\Fne$ such that
\begin{align*}
\mu(W^{u_{i}}(z))=X^{i}(z),\quad
\conf(X^{i})=j_{i}+1
\end{align*}
for all $i$.
\end{cor}

Though the following assertion is proved in \cite{A} for regular nilpotent elements $f$, the same proof applies for arbitrary $f$ (see Section 5.9 of \cite{A} for the details).

\begin{lemma}\label{Miura1 lem}
The Miura map $\mu$ is injective for all $k$ if $\g$ is a simple Lie algebra.
\end{lemma}

We note that the analogue of Lemma \ref{Miura1 lem} for the classical $\W$-algebras is proved in \cite{DKV} and this can be used to give an yet another proof of Lemma \ref{Miura1 lem}.

 Recall that $\hC_{\X}$ is a vertex superalgebra over $\C[\X]$ generated by $\bJ^{u}(z),\varphi^{\alpha}(z),\bPhi_{\beta}(z)$ for all $u\in\g_{\leq0},\alpha\in\Dp,\beta\in\Dn$ and also a chain complex over $\C[\X]$ with the differential $d_{(0)}$ (see Section \ref{classical subsec} for definitions). Let
\begin{align*}
\Fne_{\X},\quad
V^{\tau_{\X}}(\g_{\leq0}),
\quad
V^{\tau_{\X}}(\g_{0})
\end{align*}
be the vertex subalgebras of $\hC_{\X}$ over $\C[\X]$ generated by
\begin{align*}
\bPhi_{\alpha}(z),\quad
\bJ^{u}(z),\quad
\bJ^{v}(z).
\end{align*}
for all $\alpha\in\Dn, u\in\g_{\leq0}, v\in\g_{0}$. By Kac and Wakimoto arguments in the case of $C_{k}$, one finds that the cohomology $H(\hC_{\X},d_{(0)})$ is a free $\C[\X]$-module vertex superalgebra generated by
\begin{align*}
W^{u_{i}}_{\X}(z)=\bJ^{u_{i}}(z)+(\mathrm{lower}\ \mathrm{terms})
\end{align*}
for all $i$ in $V^{\tau_{\X}}(\g_{\leq0})\otimes\Fne_{\X}$ and we call it {\it the $\W$-algebra over $\C[\X]$} denoted by
\begin{align*}
\W^{\X}(\g,f;\Gamma)=H(\hC_{\X},d_{(0)}).
\end{align*}
By construction, if $k+h^{\vee}\neq0$,
\begin{align*}
\W^{k}(\g,f;\Gamma)=\W^{\X}(\g,f;\Gamma)\underset{\C[\X]}{\otimes}\C_{\p},
\end{align*}
where $\C_{\p}$ is a one-dimensional $\C[\X]$-module defined by $\X=\p$ and $\p=(k+h^{\vee})^{-1}$. We can also define the Miura map
\begin{align*}
\mu_{\X}:\W^{\X}(\g,f)\rightarrow V^{\tau_{\X}}(\g_{0})\otimes\Fne_{\X}
\end{align*}
in the same way and we have
\begin{align}\label{Miura eq1}
\mu=\mu_{\X}\underset{\C[\X]}{\otimes}\mathrm{ev}_{\X=\p}
\end{align}
for all $k\neq-h^{\vee}$, where $\mathrm{ev}_{\X=\p}$ is the evaluation map defined by $\X\mapsto\p$.

\begin{lemma}\label{Miura2 lem}
The Miura map $\mu_{\X}$ is injective.
\end{lemma}
\begin{proof}
 We extend the weight filtration on $C_{k}$ to $\hC_{\X}$ by $\deg\X=0$. This filtration gives the convergent spectral sequence $\widehat{E}_{n}$ on $\hC_{\X}$. By construction,
\begin{align*}
\widehat{E}_{1}=H(\hpC_{\X},{\dst}_{(0)})\otimes\Fne_{\X}.
\end{align*}
By Corollary \ref{CX cor},
\begin{align*}
H(\hpC_{\X},{\dst}_{(0)})= V^{\tau_{\X}}(\g_{0})\otimes H(\g_{+},\C[\X]).
\end{align*}
Therefore, in the same way as Theorem \ref{main1 thm}, one finds that the $\W$-algebra $\W^{\X}(\g,f;\Gamma)$ over $\C[\X]$ is isomorphic to the vertex subalgebra over $\C[\X]$ of $V^{\tau_{\X}}(\g_{0})\otimes\Fne_{\X}$, which is the intersection of kernels of the screening operators:
\begin{align}
\label{WX eq}
 \W^{\X}(\g,f;\Gamma)\simeq\bigcap_{[\beta]\in[\Pp]}\Ker Q_{[\beta],\X},
\end{align}
where
\begin{align*}
&Q_{[\beta],\X}=\sum_{\alpha\in[\beta]}\int:S^{\alpha}_{\X}(z)\bPhi_{\alpha}(z):dz\quad\mathrm{if}\ [\beta]\in[\Pp_{\frac{1}{2}}],\\
&Q_{[\beta],\X}=\sum_{\alpha\in[\beta]}\chi(e_{\alpha})\int S^{\alpha}_{\X}(z)dz\hspace{8.5mm}\mathrm{if}\ [\beta]\in[\Pp_{1}],
\end{align*}
and $S^{\alpha}_{\X}(z)$ is defined in Proposition \ref{Screening prop} by replacing $\p=(k+h^{\vee})^{-1}$ by $\X$. This implies that the Miura map $\mu_{\X}$ is injective.
\end{proof}

The following lemma is useful in applications.

\begin{lemma}\label{Miura3 lem}
Let
\begin{align*}
X^{i}_{\X}=\mu_{\X}(W^{u_{i}}_{\X})
\end{align*}
for all $i$. Then
\begin{align*}
X^{i}_{\X}|_{\X=\p}=\p\cdot\mu(W^{u_{i}})
\end{align*}
for all $k\neq-h^{\vee}$.
\end{lemma}
\begin{proof}
Compare the leading term of $W^{u_{i}}_{\X}(z)$ with that of $W^{u_{i}}(z)$. The assertion follows from the definition $\bJ^{u_{i}}(z)=\p J^{u_{i}}(z)$.
\end{proof}

\section{Applications}\label{application sec}

\subsection{$\W B_{n}$-algebras}\label{WBn subsec}
Let $n\in\Z_{\geq1}$ and $\h$ be a Cartan subalgebra of a simple Lie algebra of type $B_{n}$ with the non-degenerate symmetric bilinear form $\inv$ such that $(\theta|\theta)=2$. Denote by $\{\alpha_{i}\}_{i=1}^{n}$ a set of simple roots of type $B_{n}$ such that the simple root $\alpha_{n}$ is short.

Let $\Hi$ be the Heisenberg vertex algebra associated with $\h^{*}$ i.e. the vertex algebra $\Hi$ is generated by the even fields $\alpha_{i}(z)$ for all $i=1,\ldots,n$ satisfying
\begin{align*}
[{\alpha_{i}}_{\la}\alpha_{j}]=(\alpha_{i}|\alpha_{j})\la,
\end{align*}
and $\F$ be the vertex superalgebra generated by the odd field $\Psi(z)$ satisfying
\begin{align*}
[\Psi_{\la}\Psi]=1.
\end{align*}
Choose $\gamma\in\C$. Let $G(z)$ be an odd field on $\Hi\otimes\F$ defined by
\begin{align*}
G(z)=:(\gamma\der+b_{1}(z))(\gamma\der+b_{2}(z))\cdots(\gamma\der+b_{n}(z))\Psi(z):,
\end{align*}
where
\begin{align*}
b_{i}(z)=\sum_{j=i}^{n}\alpha_{j}(z)
\end{align*}
for $i=1,\ldots,n$, and $W_{i}(z)$ is an even field on $\Hi\otimes\F$ for $i=0,\ldots,2n-2$ defined by
\begin{align*}
[G_{\la}G]=W_{0}+\sum_{i=1}^{n-1}\gamma_{i}(W_{2i-1}\frac{\la^{2i-1}}{(2i-1)!}+W_{2i}\frac{\la^{2i}}{(2i)!})+\gamma_{n}\frac{\la^{2n}}{(2n)!},
\end{align*}
where
\begin{align*}
\gamma_{i}=\prod_{j=1}^{i}(1-2j(2j-1)\gamma^{2}).
\end{align*}

\begin{lemma}\label{WB0 lem}
Let $C_{2}(\Hi\otimes\F)$ be the subspace of $\Hi\otimes\F$ spanned by $a_{(-2)}b$ for all $a,b\in\Hi\otimes\F$. Then
\begin{align*}
[G_{\la}G]\ \equiv\ W_{0}+\sum_{i=1}^{n-1}\gamma_{i}\frac{W_{2i}\la^{2i}}{(2i)!}+\gamma_{n}\frac{\la^{2n}}{(2n)!}\quad\pmod{C_{2}(\Hi\otimes\F)}
\end{align*}
and
\begin{align*}
W_{2i}\ \equiv\sum_{1\leq j_{1}<\cdots <j_{n-i}\leq n}:b^{2}_{j_{1}}\cdots b^{2}_{j_{n-i}}:\quad\pmod{C_{2}(\Hi\otimes\F)}
\end{align*}
for all $i=0,\ldots,n-1$.
\end{lemma}
\begin{proof}
Notice that
\begin{align*}
[{b_{i}}_{\la}b_{j}]=\delta_{i,j}\la
\end{align*}
for all $i,j$. If $n=1$,
\begin{align*}
[G_{\la}G]
&=\ :b_{1}^{2}:+\gamma\der b_{1}+:(\der\Psi)\Psi:+(1-2\gamma^{2})\frac{\la^{2}}{2}\\
&\equiv\ :b_{1}^{2}:+(1-2\gamma^{2})\frac{\la^{2}}{2}\quad\pmod{C_{2}(\Hi\otimes\F)}.
\end{align*}
Let $n\geq2$ and
\begin{align*}
G'(z)=\ :(\gamma\der+b_{2}(z))\cdots(\gamma\der+b_{n}(z))\Psi(z):.
\end{align*}
Then
\begin{align*}
G(z)=\ :(\gamma\der+b_{1}(z))G'(z):.
\end{align*}
By inductions on $n$, we can assume that $G'(z)$ satisfies our assertions. Hence
\begin{align*}
[G'_{\la}G']\ \equiv\ W'_{0}+\sum_{i=1}^{n-2}\gamma_{i}\frac{W'_{2i}\la^{2i}}{(2i)!}+\gamma_{n-1}\frac{\la^{2n-2}}{(2n-2)!}\quad\pmod{C_{2}(\Hi\otimes\F)}
\end{align*}
and
\begin{align*}
W'_{2i}\ \equiv\sum_{2\leq j_{2}<\cdots <j_{n-i}\leq n}:b_{j_{2}}^{2}\cdots b_{j_{n-i}}^{2}:\quad\pmod{C_{2}(\Hi\otimes\F)}.
\end{align*}
Therefore
\begin{align*}
&[G_{\la}G]\equiv\ -\gamma^{2}\la^{2}[G'_{\la}G']-\gamma\la[G'_{\la}:b_{1}G':]+\gamma\la[G'_{\la}:b_{1}G':]+[:b_{1}G':_{\la}:b_{1}G':]\\
&\hspace{95mm}\pmod{C_{2}(\Hi\otimes\F)}\\
&\equiv\ :b_{1}^{2}W'_{0}:+\sum_{i=1}^{n-2}\gamma_{i}\frac{(:b_{1}^{2}W'_{2i}:+W'_{2i-2})\la^{2i}}{(2i)!}+\gamma_{n-1}\frac{(:b_{1}^{2}:+W'_{2n-4})\la^{2n-2}}{(2n-2)!}+\gamma_{n}\frac{\la^{2n}}{(2n)!}\\
&\hspace{95mm}\pmod{C_{2}(\Hi\otimes\F)}.
\end{align*}
Furthermore
\begin{align*}
W_{0}&\ \equiv\ :b_{1}^{2}W'_{0}:\ \equiv\ :b_{1}^{2}\cdots b_{n}^{2}:&\pmod{C_{2}(\Hi\otimes\F)},\\
W_{2n-2}&\ \equiv\ :b_{1}^{2}:+W'_{2n-4}\ \equiv\ :b_{1}^{2}:+\cdots+:b_{n}^{2}:&\pmod{C_{2}(\Hi\otimes\F)},\\
W_{2i}&\ \equiv\ :b_{1}^{2}W'_{2i}:+W'_{2i-2}\ \equiv\ \sum_{1\leq j_{1}<\cdots <j_{n-i}\leq n}:b_{j_{1}}^{2}\cdots b_{j_{n-i}}^{2}:&\pmod{C_{2}(\Hi\otimes\F)}
\end{align*}
for $i=1,\ldots,n-2$ by our assumptions. This completes the proof.
\end{proof}

The vertex subalgebra of $\Hi\otimes\F$ generated by the odd field $G(z)$ and the even fields $W_{2i}(z)$ for $i=0,\ldots,n-1$ is {\it the $\W B_{n}$-algebra for $\gamma\in\C$} introduced by Fateev and Lukyanov \cite{FL}. Let $\gamma_{\pm}\in\C$ such that $\gamma_{+}+\gamma_{-}=\gamma$ and $\gamma_{+}\gamma_{-}=-1$. Define the operators
\begin{align*}
Q_{i}:\Hi\otimes\F\rightarrow\Hi_{\gamma_{+}\alpha_{i}}\otimes\F
\end{align*}
for $i=1,\ldots,n$ by
\begin{align*}
&Q_{i}=\int \e^{\gamma_{+}\int\alpha_{i}(z)}\ dz\quad\mathrm{for}\quad i\neq n,\\
&Q_{n}=\int:\e^{\gamma_{+}\int\alpha_{n}(z)}\Psi(z): dz,
\end{align*}
where $\e^{\int\beta(z)}$ for $\beta\in\h^{*}$ is defined in Section \ref{main thm sec}. According to \cite{FL},
\begin{align*}
\W B_{n}=\bigcap_{i=1}^{n}\Ker Q_{i}
\end{align*}
for generic $\gamma$.

\begin{lemma}
\label{WB1 lem}
The odd vector $G$ and even vectors $W_{2i}$ for $i=0,\ldots,n-1$ belong to the subspace $\displaystyle\bigcap_{i=1}^{n}\Ker Q_{i}$ of $\Hi\otimes\F$.
\end{lemma}
\begin{proof}
It is enough to show that $Q_{i}\cdot G=0$ for all $i$. Let
\begin{align*}
&S^{\alpha_{i}}(z)=\e^{\gamma_{+}\int\alpha_{i}(z)},\\
&G_{i}(z)=\ :(\gamma\der+b_{i}(z))\cdots(\gamma\der+b_{n}(z))\Psi(z):
\end{align*}
for $i=1,\ldots,n$ and $G_{n+1}(z)=\Psi(z)$. Then, for $i=1,\ldots,n-1$,
\begin{align*}
[{S^{\alpha_{i}}}_{\la}G]_{\la=0}
=\ :(\gamma\der+b_{1})\cdots(\gamma_{+}+b_{i-1})[{S^{\alpha_{i}}}_{\la}:(\gamma\der+b_{i})(\gamma\der+b_{i+1})G_{i+2}:]_{\la=0}:
\end{align*}
and
\begin{align*}
&[{S^{\alpha_{i}}}_{\la}:(\gamma\der+b_{i})(\gamma\der+b_{i+1})G_{i+2}:]_{\la=0}\\
&=\gamma\gamma_{+}:\der S^{\alpha_{i}}G_{i+2}:-\gamma_{+}:S^{\alpha_{i}}b_{i+1}G_{i+2}:+\gamma_{+}:b_{i}S^{\alpha_{i}}G_{i+2}:\\
&=\gamma_{+}:(\gamma\der S^{\alpha_{i}}-\gamma_{+}\der S^{\alpha_{i}}-b_{i+1}S^{\alpha_{i}}+b_{i}S^{\alpha_{i}})G_{i+2}:\\
&=\ :(-\der S^{\alpha_{i}}+\gamma_{+}\alpha_{i}S^{\alpha_{i}})G_{i+2}:\\
&=0.
\end{align*}
Hence
\begin{align*}
Q_{i}\cdot G=0\quad \mathrm{for}\ i=1,\ldots,n-1.
\end{align*}
Moreover
\begin{align*}
[:S^{\alpha_{n}}\Psi:_{\la}G]=\ :(\gamma\der+b_{1})\cdots(\gamma_{+}+b_{n-1})[:S^{\alpha_{n}}\Psi:_{\la}:(\gamma\der+b_{n})\Psi:]_{\la=0}:
\end{align*}
and
\begin{align*}
[:S^{\alpha_{n}}\Psi:_{\la}:(\gamma\der+b_{n})\Psi:]_{\la=0}
&=\gamma\der S^{\alpha_{n}}-\gamma_{+}\der S^{\alpha_{n}}+:b_{n}S^{\alpha_{n}}:\ =0.
\end{align*}
Therefore
\begin{align*}
Q_{n}\cdot G=0.
\end{align*}
This completes the proof.
\end{proof}

\begin{lemma}
\label{WB2 lem}
If $k$ is generic,
\begin{align*}
\W^{k}(\osp(1,2n),f_{reg};\Gamma)\simeq\bigcap_{i=1}^{n}\Ker Q_{i},
\end{align*}
where $\gamma_{+}=-1/\sqrt{2k+2n+1}$, $f_{reg}$ is a regular nilpotent element of the even part of $\osp(1,2n)$ and $\Gamma$ is the Dynkin grading of $f_{reg}$.
\end{lemma}

\begin{proof}
Let $\g=\osp(1,2n)$ and $\{\beta_{i}\}_{i=1}^{n}$ be a set of simple roots of $\g$ such that $\beta_{n}$ is an odd root and $\{\beta_{1},\ldots,\beta_{n-1},2\beta_{n}\}$ is a set of simple roots of the even part of $\g$. Note that the even part of $\g$ is the symplectic Lie algebra $\spf(2n)$ of rank $n$. Let $f_{reg}$ be a regular nilpotent element of the even part of $\g$ defined by
\begin{align*}
f_{reg}=e_{-\beta_{1}}+\cdots+e_{-\beta_{n-1}}+e_{-2\beta_{n}}
\end{align*}
and $\{e,h,f_{reg}\}$ be a $\slf_{2}$-triple. We set $x=\frac{1}{2}h$ and denote by $\inv_{\g}$ the normalized even supersymmetric non-degenerate invariant bilinear form on $\g$. Notice that the dual Coxeter number $h^{\vee}$ of $\g$ is equal to $n+\frac{1}{2}$. A semisimple element $x$ defines the Dynkin grading $\Gamma$ on $\g$ with respect to $f_{reg}$ denoted by
\begin{align*}
\Gamma:\g=\bigoplus_{j\in\frac{1}{2}\Z}\g_{j},
\end{align*}
where $\g_{j}$ is the eigenspace of $\ad x$ with an eigenvalue $j$. Then
\begin{align*}
\g_{0}=\h',\quad
\Pi_{\frac{1}{2}}=\{\beta_{n}\},\quad
\Pi_{1}=\{\beta_{1},\ldots,\beta_{n-1}\},
\end{align*}
where $\h'$ is a Cartan subalgebra of $\g$ containing $x$. The subspace $\g_{\frac{1}{2}}$ of $\g$ is the one-dimensional odd space spanned by $e_{\beta_{n}}$. The neutral free superfermion vertex superalgebra associated with $\g_{\frac{1}{2}}$ is isomorphic to $\F$ since $\chi([e_{\beta_{n}},e_{\beta_{n}}])\neq0$. Let $\Hi'$ be the Heisenberg vertex algebra associated with $\h'$. If $k$ is generic, then by Theorem \ref{main2 thm},  the vertex superalgebra $\W^{k}(\g,f_{reg};\Gamma)$ is isomorphic to the vertex subalgebra of $\Hi'\otimes\F$, which is the intersection of kernels of the screening operators:
\begin{align*}
\W^{k}(\g,f_{reg};\Gamma)\simeq\bigcap_{i=1}^{n-1}\Ker\int \e^{-\frac{1}{\kh'}\int\beta_{i}(z)}\ dz\cap
\Ker\int:\e^{-\frac{1}{\kh'}\int\beta_{n}(z)}\Psi(z):dz,
\end{align*}
where $\kh'=\sqrt{\mathstrut k+n+\frac{1}{2}}$. Let
\begin{align*}
\alpha_{i}=\sqrt{2}\beta_{i}\in{\h'}^{*}\quad\mathrm{for}\ i=1,\ldots,n.
\end{align*}
Then $(\alpha_{i}|\alpha_{j})_{\g}=(\alpha_{i}|\alpha_{j})$ for all $i,j$ and the vertex algebra $\Hi'$ is isomorphic to $\Hi$. Moreover
\begin{align*}
\e^{-\frac{1}{\kh'}\int\beta_{i}(z)}=\e^{\gamma_{+}\int\alpha_{i}(z)}
\end{align*}
for all $i$, where $\gamma_{+}=-1/\sqrt{2}\kh'$. Therefore our assertion follows.
\end{proof}

\begin{theorem}\label{WBn thm}
If $k+n+\frac{1}{2}\neq0$,
\begin{align*}
\W^{k}(\osp(1,2n),f_{reg};\Gamma)\simeq\W B_{n},
\end{align*}
where $\gamma=\kh-\frac{1}{\kh}$, $\kh=\sqrt{2k+2n+1}$, $f_{reg}$ is a regular nilpotent element of the even part of $\osp(1,2n)$ and $\Gamma$ is the Dynkin grading of $f_{reg}$. Moreover the fields
\begin{align*}
G(z),\quad W_{2i}(z)\ (i=0,\ldots,n-1)
\end{align*}
are strongly generating fields of $\W^{k}(\osp(1,2n),f_{reg};\Gamma)$ with conformal weights $\conf(G)=n+\frac{1}{2}$ and $\conf(W_{2i})=2n-2i$.
\end{theorem}

\begin{proof}
Let
\begin{align*}
f=f_{reg},\quad
\g=\osp(1,2n).
\end{align*}
If $k$ is generic, by Lemma \ref{WB2 lem},
\begin{align}\label{WB eq1}
\W^{k}(\g,f;\Gamma)\simeq\bigcap_{i=1}^{n}\Ker Q_{i},
\end{align}
where $\gamma_{+}=-1/\sqrt{2k+2n+1}$, which describes the image of the Miura map $\mu$ (see Section \ref{Miura sec} for the definition). The conformal weights of $\Hi\otimes\F$ induced by this isomorphism is determined by
\begin{align*}
\conf(\alpha_{i})=1,\quad
\conf(\Psi)=\frac{1}{2}.
\end{align*}
By Lemma \ref{WB1 lem}, the odd vector $G$ and $W_{2i}$ for $i=0,\ldots,n-1$ belong to $\bigcap_{i=1}^{n}\Ker\ Q_{i}$ and
\begin{align*}
\conf(G)=n+\frac{1}{2},\quad \conf(W_{2i})=2n-2i.
\end{align*}
By definition, the $\W B_{n}$-algebra is generated by $G(z)$ and $W_{2i}(z)$ for all $i$ as a vertex subalgebra in $\Hi\otimes\F$. Since $W_{2i}(z)$ has a leading term
\begin{align*}
W_{2i}(z)=\sum_{1\leq j_{1}<\cdots <j_{n-i}\leq n}:b^{2}_{j_{1}}(z)\cdots b^{2}_{j_{n-i}}(z):+\cdots
\end{align*}
for $i=0,\ldots,n-1$ by Lemma \ref{WB0 lem}, there exists a set of strongly generating fields of the $\W B_{n}$-algebras such that $G(z)$ and $W_{2i}(z)$ are included in this set. Let $\g_{ev}$ be the even part of $\g$ and $\g_{od}$ be the odd part of $\g$. Denote by $\g^{f}$ the centralizer of $f$ in $\g$. Then
\begin{align*}
\g^{f}=\g_{ev}^{f}\oplus\g_{od}^{f}.
\end{align*}
Since $\g_{ev}=\spf(2n)$ and $f$ is a regular nilpotent element of $\g_{ev}$,
\begin{align*}
\dim\g_{ev}^{f}=n
\end{align*}
and there exists a basis $\{u_{i}\}_{i=1}^{n}$ of $\g_{ev}^{f}$ such that
\begin{align*}
u_{i}\in\g_{-j_{i}},\quad
j_{i}=2i-1
\end{align*}
for $i=1,\ldots,n$, where $j_{i}$ is called the $i$-th exponent of $\spf(2n)$. A subspace $\g_{od}^{f}$ of $\g^{f}$ is a one-dimensional odd space spanned by 
\begin{align*}
e_{-\beta_{1}-\cdots-\beta_{n-1}-\beta_{n}}\in\g_{-n+\frac{1}{2}},
\end{align*}
where we follow the notation in the proof of Lemma \ref{WB2 lem}. Therefore, by Corollary \ref{Miura cor}, the vertex superalgebra $\W^{k}(\g,f;\Gamma)$ is generated by an odd field $G'(z)$ and even fields $X^{i}(z)$ with conformal weights
\begin{align*}
\conf(G')=n+\frac{1}{2},\quad\conf(X^{i})=j_{i}+1=2i
\end{align*}
for $i=1,\ldots,n$. Therefore, for generic $k$, we can choose $X^{i}(z)$ and $G'(z)$ such that
\begin{align}\label{WB eq2}
\mu(G')=G,\quad\mu(X^{i})=W_{2n-2i}.
\end{align}
Hence the isomorphism \eqref{WB eq1} implies that the Miura map
\begin{align}\label{WB eq3}
\mu:\W^{k}(\g,f;\Gamma)\rightarrow\W B_{n}
\end{align}
is an isomorphism for generic $k$. Since the equations \eqref{WB eq2} hold not only for generic $k$ but also for $k\neq-n-\frac{1}{2}$ by Lemma \ref{Miura3 lem}, the Miura map \eqref{WB eq3} is also an isomorphism for $k\neq-n-\frac{1}{2}$. This completes the proof.
\end{proof}

\begin{rem}
For $k\neq-h^{\vee}(=-n-\frac{1}{2})$, the $\W$-algebra $\W^{k}(\osp(1,2n),f_{reg};\Gamma)$ is isomorphic to the vertex subalgebra of $\Hi\otimes\F$ generated by $G(z)$. Therefore the Feigin-Frenkel duality
\begin{align*}
\W^{k}(\osp(1,2n),f_{reg};\Gamma)\simeq\W^{k'}(\osp(1,2n),f_{reg};\Gamma)
\end{align*}
holds for this case, where $4(k+h^{\vee})(k'+h^{\vee})=1$.
\end{rem}

\subsection{$\W^{(2)}_{n}$-algebras}\label{W2n sec}
Let $n\in\Z_{\geq 2}$, $k\in\C$ such that $k+n\neq0$ and $V$ be a vector space spanned by
\begin{align*}
a_{n-1},\ a_{n-2},\ldots,\ a_{1},\ \psi,\ \xi
\end{align*}
with non-degenerate $\C$-bilinear form $\inv_{V}:V\times V\rightarrow\C$ defined by the following Gram matrix:
\begin{align*}
\bordermatrix{
          & a_{n-1} & a_{n-2} & a_{n-3}  & \cdots & a_{1}    & \psi    & \xi       \cr
a_{n-1} & 2(k+n)  & -k-n    & 0         & \cdots & 0        & 0        &  0        \cr
a_{n-2} & -k-n    & 2(k+n)  & -k-n     & \cdots & 0         & 0        & 0         \cr
\vdots & \vdots & \vdots & \vdots  & \ddots & \vdots & \vdots & \vdots \cr
a_{1}    & 0        & 0         & 0          & \cdots & 2(k+n)  & -k-n    & 0         \cr
\psi     &  0      & 0          & 0         & \cdots & -k-n    & 1         & 1         \cr
\xi      & 0       & 0          & 0          & \cdots & 0         & 1         & 0          \cr
}_{.}
\end{align*}
Let $\Hi$ be the Heisenberg vertex algebra associated with $V$ and $\inv_{V}$, and $\Hi_{m\xi}$ be the $\Hi$-module with the highest weight $m\xi\in V$ for $m\in\Z$. Denote by $e^{m\xi}$ the highest weight vector of $\Hi_{m\xi}$. Then
\begin{align*}
\V_{\xi}=\bigoplus_{m\in\Z}\Hi_{m\xi}
\end{align*}
has a vertex algebra structure such that the vertex operator of $e^{m\xi}$ is defined by
\begin{align*}
e^{m\xi}(z)=Y(e^{m\xi},z):=\e^{m\int\xi(z)}
\end{align*}
for all $m\in\Z$. Since $(\xi\mid\xi)_{V}=0$, the parity of $e^{m\xi}(z)$ is even for all $m\in\Z$ and the following formula holds:
\begin{align*}
[{e^{m\xi}}_{\la}\e^{m'\xi}]=0
\end{align*}
for all $m,m'\in\Z$. The $\W^{(2)}_{n}$-algebra is defined as a vertex subalgebra of $\V_{\xi}$ generated by
\begin{align*}
\E(z)=e^{\xi}(z),\quad
\F(z)=:\Po(z) e^{-\xi}(z):,
\end{align*}
where
\begin{align*}
\Po=-((k+n-1)\der+\psi+\sum_{i=1}^{n-1}a_{i})\cdot((k+n-1)\der+\psi+\sum_{i=1}^{n-2}a_{i})
\cdots((k+n-1)\der+\psi+a_{1})\psi.
\end{align*}
We remark that we should substitute $\Po$ for $\F$ after expanding $\Po$ formally. For example, if $n=2$,
\begin{align*}
&\Po=-(k+1)\der\psi-(\psi+a_{1})\psi,\\
&\F(z)=-(k+1):(\der\psi)(z)\e^{-\xi}(z):-:(\psi(z)+a_{1}(z))\psi(z)\e^{-\xi}(z):.
\end{align*}
Let
\begin{align*}
A_{i}(z)=\e^{\int a_{i}(z)}\quad\mathrm{for}\ i=1,\ldots,n-1,\quad
Q(z)=\e^{\int \psi(z)}.
\end{align*}
By a result in \cite{FS},
\begin{align*}
\W^{(2)}_{n}=\bigcap_{i=1}^{n-1}\Ker\int A_{i}(z)\ dz\cap\Ker\int Q(z)\ dz
\end{align*}
for generic $k$.

For $n=2$, the vertex algebra $\W^{(2)}_{2}$ is isomorphic to the universal affine vertex algebra associated with $\slf_{2}$ at level $k$ and for $n=3$, the vertex algebra $\W^{(2)}_{3}$ is isomorphic to the Bershadsky-Polyakov algebra (\cite{B,P}).

\begin{rem}
Feigin and Semikhatov gave equivalent $n+1$ definitions of the $\W^{(2)}_{n}$-algebra, denoted by $\W^{(2)}_{n[m]}$ for $m\in\Z$ with $0\leq m\leq n$. In this paper, we use only the definition of $\W^{(2)}_{n[0]}$.
\end{rem}

\begin{lemma}
\begin{multline*}
\F(z)=-:((k+n-1)(\der+\xi(z))+\psi(z)+\sum_{i=1}^{n-1}a_{i}(z))\\
\cdots((k+n-1)(\der+\xi(z))+\psi(z)+a_{1}(z))\psi(z)\e^{-\xi}(z):.
\end{multline*}
\end{lemma}
\begin{proof}
If $n=2$,
\begin{align*}
\F(z)=&-(k+1):(\der\psi)(z)\e^{-\xi}(z):-:(\psi(z)+a_{1}(z))\psi(z)\e^{-\xi}(z):\\
=&-(k+1)(\der:\psi(z)\e^{-\xi}(z):-:\psi(z)\der\e^{-\xi}(z):)-:(\psi(z)+a_{1}(z))\psi(z)\e^{-\xi}(z):\\
=&-:((k+1)(\der+\xi(z))+\psi(z)+a_{1}(z))\psi(z)\e^{-\xi}(z):.
\end{align*}
For $n\geq3$, the above formula is proved similarly using the induction.
\end{proof}

 Assume $n\geq3$. Let $\{e_{i,j}\}_{1\leq i,j\leq n}$ be the standard basis of $\gl_{n}$ and
\begin{align*}
h_{i}=e_{i,i}-e_{i+1,i+1}
\end{align*}
for $i=1,\ldots,n-1$. Denote by
\begin{align*}
\g=\slf_{n}\subset\gl_{n}
\end{align*}
the special linear Lie algebra of rank $n$. Then $\{h_{i}\}_{i=1}^{n-1}$ is a basis of the Cartan subalgebra $\h$. Let $\epsilon_{i}$ be the dual element of $e_{i,i}$ for $i=1,\ldots,n$ and $\alpha_{j}=\epsilon_{j}-\epsilon_{j+1}$ for $j=1,\ldots,n-1$. Then $\Delta=\{\epsilon_{i}-\epsilon_{j}\}_{1\leq i\neq j\leq n}$ is a root system of $\g$ and
\begin{align*}
\Pi=\{\alpha_{1},\ldots,\alpha_{n-1}\}
\end{align*}
is a set of simple roots of $\Delta$. Choose a root vector $e_{\epsilon_{i}-\epsilon_{j}}=e_{i,j}$. Define a subregular nilpotent element $f=f_{sub}$ of $\g$ by
\begin{align*}
f_{sub}:=e_{-\alpha_{2}}+\cdots+e_{-\alpha_{n-1}}
\end{align*}
and a semisimple element $x$ of $\g$ by
\begin{align*}
x=\frac{1}{2n}\sum_{i=1}^{n-1}(n-i)(in-2)h_{i}.
\end{align*}
Then $\ad x$ defines a good grading $\Gamma$ with respect to $f$. The corresponding weighted Dynkin diagram is the following:
\begin{align}\label{Dynkin diag}
\begin{Dynkin}
\Dbloc{\Dcirc\Deast\Dtext{t}{0}}
\Dbloc{\Dcirc\Dwest\Deast\Dtext{t}{1}}
\Dbloc{\Ddots}
\Dbloc{\Dcirc\Dwest\Dtext{t}{1}}
\end{Dynkin}
\end{align}
Let $\g_{j}$ be the homogeneous subspace of $\g$ with degree $j$. Then
\begin{align*}
\g_{0}=\C e_{\alpha_{1}}\oplus\h\oplus\C e_{-\alpha_{1}},\quad
\g_{\frac{1}{2}}=0.
\end{align*}
Denote by $\Del$, $\Pp$, $[\Pp]$ the restricted root system associated with $\Gamma$, a base of $\Del$, a quotient set of $\Pp$ respectively, as defined in Section \ref{Screening sec}. We have
\begin{align*}
&\Pp=\{\alpha_{1}+\alpha_{2},\ \alpha_{2},\ldots,\ \alpha_{n-1}\},\\
&[\Pp]=\{[\alpha_{2}],\ldots,[\alpha_{n-1}]\},
\end{align*}
where
\begin{align*}
[\alpha_{2}]=\{\alpha_{1}+\alpha_{2},\ \alpha_{2}\},\quad
[\alpha_{i}]=\{\alpha_{i}\}\quad\mathrm{for}\ i>2.
\end{align*}
Let $\inv$ be the non-degenerate symmetric invariant bilinear form on $\g$ defined by $(u|v)=\mathrm{tr}(u v)$ for $u,v\in\g$. By Theorem \ref{main1 thm}, the vertex algebra $\W^{k}(\slf_{n},f_{sub};\Gamma)$ is isomorphic to the vertex subalgebra of $\Vtau$, which is the intersection of kernels of the screening operators:
\begin{align*}
\W^{k}(\slf_{n},f_{sub};\Gamma)\simeq\bigcap_{i=2}^{n-1}\Ker\int S^{\alpha_{i}}(z)\ dz
\end{align*}
for generic $k$. Let $E(z)$ and $F(z)$ be th fields on $\Vtau$ defined by
\begin{align*}
E(z)=&\ J^{e_{\alpha_{1}}}(z),\\
F(z)=&\ :((k+n-1)\der+\sum_{i=1}^{n-1}J^{h_{i}}(z))\cdot((k+n-1)\der+\sum_{i=1}^{n-2}J^{h_{i}}(z))\\
&\cdots((k+n-1)\der+J^{h_{1}}(z)+J^{h_{2}}(z))J^{e_{-\alpha_{1}}}(z):.
\end{align*}

\begin{lemma}
\begin{align*}
E,F\ \in\ \bigcap_{i=2}^{n-1}\Ker\int S^{\alpha_{i}}(z)\ dz.
\end{align*}
\end{lemma}

\begin{proof}
It is enough to show that
\begin{align*}
[{S^{\alpha_{i}}}_{\la}E]_{\la=0}=[{S^{\alpha_{i}}}_{\la}F]_{\la=0}=0
\end{align*}
for all $i=2,\ldots,n-1$. By Proposition \ref{Screening prop},
\begin{align*}
&[S^{\alpha}_{m}, J^{h}_{(n)}]=\alpha(h)S^{\alpha}_{m+n}\quad\mathrm{for}\ \mathrm{all}\ \alpha\in\Pp,\ h\in\h,\\
&[S^{\alpha_{1}+\alpha_{2}}_{m}, J^{e_{\alpha_{1}}}_{(n)}]=\ S^{\alpha_{2}}_{m+n},\quad
[S^{\alpha}_{m}, J^{e_{\alpha_{1}}}_{(n)}]=0\quad\mathrm{for}\ \mathrm{all}\ \alpha\neq\alpha_{1}+\alpha_{2},\\
&[S^{\alpha_{2}}_{m}, J^{e_{-\alpha_{1}}}_{(n)}]=\ S^{\alpha_{1}+\alpha_{2}}_{m+n},\quad
[S^{\alpha}_{m}, J^{e_{-\alpha_{1}}}_{(n)}]=0\quad\mathrm{for}\ \mathrm{all}\ \alpha\neq\alpha_{2}
\end{align*}
and
\begin{align*}
&\der S^{\alpha_{1}+\alpha_{2}}(z)=-\frac{1}{k+n}:(J^{h_{1}+h_{2}}(z)S^{\alpha_{1}+\alpha_{2}}(z)+J^{e_{-\alpha_{1}}}(z)S^{\alpha_{2}}(z)):,\\
&\der S^{\alpha_{2}}(z)=-\frac{1}{k+n}:(J^{h_{2}}(z)S^{\alpha_{2}}(z)+J^{e_{\alpha_{1}}}(z)S^{\alpha_{1}+\alpha_{2}}(z)):,\\
&\der S^{\alpha_{i}}(z)=-\frac{1}{k+n}:J^{h_{i}}(z)S^{\alpha_{i}}(z):\quad\mathrm{for}\ i=3,\ldots,n-1.
\end{align*}
Therefore
\begin{align*}
[{S^{\alpha_{i}}}_{\la}E]_{\la=0}=0
\end{align*}
for all $i$. Next, we show $[{S^{\alpha_{i}}}_{\la}F]_{\la=0}=0$ for all $i$. Let $F_{1}(z)=J^{e_{-\alpha_{1}}}(z)$ and
\begin{align*}
F_{i}(z)=\ :((k+n-1)\der+J^{b_{i}}(z))\cdots((k+n-1)\der+J^{b_{2}}(z))J^{e_{-\alpha_{1}}}(z):.
\end{align*}
for $i\geq2$, where
\begin{align*}
b_{j}=\sum_{i=1}^{j}h_{i}
\end{align*}
for all $j$. First,
\begin{align*}
[{S^{\alpha_{i}}}_{\la}F]_{\la=0}
=\ :((k+n-1)\der+J^{b_{n-1}}(z))\cdots((k+n-1)\der+J^{b_{i+1}}(z))[{S^{\alpha_{i}}}_{\la}F_{i}]_{\la=0}:.
\end{align*}
Therefore it is enough to show that
\begin{align*}
[{S^{\alpha_{i}}}_{\la}F_{i}]_{\la=0}=0
\end{align*}
for all $i$. For $i=2$, we have
\begin{align*}
[{S^{\alpha_{2}}}_{\la}F_{2}]_{\la=0}=&\ (k+n-1)\der S^{\alpha_{1}+\alpha_{2}}+:J^{b_{2}}S^{\alpha_{1}+\alpha_{2}}:+:S^{\alpha_{2}}J^{e_{-\alpha_{2}}}:\\
=&\ (k+n)\der S^{\alpha_{1}+\alpha_{2}}+:(J^{b_{2}}S^{\alpha_{1}+\alpha_{2}}+J^{e_{-\alpha_{2}}}S^{\alpha_{2}}):\\
=&\ 0.
\end{align*}
For $i\geq3$, since $[{S^{\alpha_{i}}}_{\la}F_{i-2}]_{\la=0}=0$, we obtain that
\begin{align*}
[{S^{\alpha_{i}}}_{\la}F_{i}]_{\la=0}=&\ [{S^{\alpha_{i}}}_{\la}:((k+n-1)\der+J^{b_{i}})((k+n-1)\der+J^{b_{i-1}})F_{i-2}:]_{\la=0}\\
=&-(k+n-1):\der S^{\alpha_{i}}F_{i-2}:-:J^{b_{i}}S^{\alpha_{i}}F_{i-2}:+:S^{\alpha_{i}}J^{b_{i-1}}F_{i-2}:\\
=&\ -:((k+n)\der S^{\alpha_{i}}+J^{h_{i}}S^{\alpha_{i}})F_{i-2}:\\
=&\ 0.
\end{align*}
This completes the proof.
\end{proof}

\begin{theorem}\label{W2n thm}
If $k+n\neq0$,
\begin{align*}
\W^{k}(\slf_{n}, f_{sub};\Gamma)\simeq\W^{(2)}_{n},
\end{align*}
where $f_{sub}$ is a subregular nilpotent element of $\slf_{n}$ and $\Gamma$ is the good grading such that the corresponding weighted Dynkin diagram is \eqref{Dynkin diag}.
\end{theorem}

\begin{proof}
Define a vertex algebra homomorphism
\begin{align*}
\pi\ :\ V^{\tau_{k}}(\g_{0})\longrightarrow\V_{\xi}
\end{align*}
by
\begin{align*}
J^{h_{1}}(z)&\quad\longmapsto\quad (k+n-2)\xi(z)+2\psi(z)+a_{1}(z),\\
J^{h_{2}}(z)&\quad\longmapsto\quad \xi(z)-\psi(z)+a_{2}(z),\\
J^{h_{i}}(z)&\quad\longmapsto\quad a_{i}(z)\hspace{8mm}(i=3,\ldots,n-1),\\
J^{e_{\alpha_{1}}}(z)&\quad\longmapsto\quad\e^{\int\xi(z)},\\
J^{e_{-\alpha_{1}}}(z)&\quad\longmapsto\quad-:((k+n-1)(\der+\xi(z))+\psi(z)+a_{1}(z))\psi(z)\e^{-\int\xi(z)}:.
\end{align*}
Then
\begin{align*}
\Img(\pi)\subset\Ker\int A_{1}(z)\ dz\cap\Ker\int Q(z)\ dz.
\end{align*}
Consider the Lie algebra decomposition
\begin{align*}
\g_{0}=\slf_{2}\oplus\z,
\end{align*}
where $\slf_{2}$ is the subalgebra of $\g_{0}$ spanned by $e_{\alpha_{1}},h_{1},e_{-\alpha_{1}}$ and $\z$ is the center of $\g_{0}$. If we restrict $\pi$ to the vertex subalgebra $V^{k+n-2}(\slf_{2})$ of $\Vtau$,  the map $\pi$ coincides with the vertex algebra homomorphism of the Wakimoto construction of $V^{k+n-2}(\slf_{2})$, which is injective (\cite{Fre}). Since
\begin{align*}
\Vtau=V^{k+n-2}(\slf_{2})\otimes V^{k+n}(\z)
\end{align*}
(see Remark \ref{Vtau rem}) and $V^{k+n}(\z)$ is the Heisenberg vertex algebra associated with $\z$, $\pi$ is injective and we can regard $\Vtau$ as a vertex subalgebra of $\V_{\xi}$ via $\pi$. Then $S^{\alpha_{i}}(z)$ acts on $\Vtau$ as
\begin{align*}
&S^{\alpha_{1}+\alpha_{2}}(z)=-:Q(z)\e^{-\frac{1}{k+n}\int a_{2}(z)-\int\xi(z)}:,\\
&S^{\alpha_{i}}=\e^{-\frac{1}{k+n}\int a_{i}(z)}\quad(i=2,\ldots,n-1).
\end{align*}
Therefore, for generic $k$,
\begin{align*}
\pi(\W^{k}(\slf_{n},f_{sub};\Gamma))
\subset&\bigcap_{i=2}^{n-1}\Ker\int S^{\alpha_{i}}(z)\ dz\cap\Ker\int A_{1}(z)\ dz\cap\Ker\int Q(z)\ dz\\
=&\bigcap_{i=2}^{n-1}\Ker\int A_{i}(z)\ dz\cap\Ker\int A_{1}(z)\ dz\cap\Ker\int Q(z)\ dz\\
=&\W^{(2)}_{n}
\end{align*}
because
\begin{align*}
\Ker\int S^{\alpha_{i}}(z)\ dz=\Ker\int A_{i}(z)\ dz
\end{align*}
for all $i=2,\ldots,n-1$. Moreover
\begin{align*}
\pi(E)=\mathcal{E},\quad
\pi(F)=\mathcal{F}
\end{align*}
and these are generating fields of $\W^{(2)}_{n}$ for all $k+n\neq0$. Hence $\pi|_{\W^{k}(\slf_{n},f_{sub})}$ induces a vertex algebra isomorphism
\begin{align*}
\W^{k}(\slf_{n},f_{sub})\simeq\W^{(2)}_{n}
\end{align*}
for generic $k$. By Lemma \ref{Miura1 lem}, this isomorphism holds for all $k+n\neq0$.
\end{proof}

\begin{rem}
For a subregular nilpotent element $f_{sub}$ in $\g=\slf_{n}$, there exists a good grading $\Gamma$ such that $\g_{0}=\h$. Thus, Theorem \ref{main2 thm} gives a free field realization of the $\W^{(2)}_{n}$-algebra.
\end{rem}


\begin{thebibliography}{ACGHR}
\bibitem[ACGHR]{ACGHR}H. R. Afshar, T. Creutzig, D. Grumiller, Y. Hikida, P. B. R\o nne.
\newblock Unitary W-algebras and three-dimensional higher spin gravities with spin one symmetry.
\newblock {\em J. High Energy Phys.}, (6), 063, 2014.

\bibitem[A1]{A05}T. Arakawa.
\newblock Representation theory of superconformal algebras and the Kac-Roan-Wakimoto Conjecture.
\newblock {\em Duke Math. J.}, 130(3):435--478, 2005.

\bibitem[A2]{A07}T. Arakawa.
\newblock Representation theory of $\mathscr{W}$-algebras.
\newblock {\em Invent. Math.}, 169(2):219--320, 2007.

\bibitem[A3]{A}T. Arakawa.
\newblock Introduction to W-algebras and their representation theory.
\newblock arXiv:1605.00138.

\bibitem[AKM]{AKM}T. Arakawa, T. Kuwabara, F. Malikov.
\newblock Localization of affine W-algebras.
\newblock {\em Comm. Math. Phys.}, 335(1):143--182, 2015.

\bibitem[Ber]{B}M. Bershadsky.
\newblock Conformal field theories via Hamiltonian reduction.
\newblock {\em Comm. Math. Phys.}, 139(1):71--82, 1991. 

\bibitem[BG]{BG}J. Brundan, S. M. Goodwin.
\newblock Good grading polytopes.
\newblock {\em Proc. Lond. Math. Soc., (3)}, 94(1):155--180, 2007.

\bibitem[CM]{CM}D. H. Collingwood,  W. M. McGovern.
\newblock Nilpotent orbits in semisimple Lie algebras.
\newblock Van Nostrand Reinhold Mathematics Series, {\em Van Nostrand Reinhold Co., New York}, 1993.

\bibitem[DK]{DK}A. De Sole, V. G. Kac.
\newblock Finite vs affine $W$-algebras.
\newblock {\em Jpn. J. Math.}, 1(1):137--261, 2006.

\bibitem[DKV]{DKV}A. De Sole, V. G. Kac, D. Valeri.
\newblock Structure of classical (finite and affine) W-algebras.
\newblock {\em J. Eur. Math. Soc.}, 18(9):1873--1908, 2016.

\bibitem[EK]{EK}A. G. Elashvili, V. G. Kac.
\newblock Classification of good gradings of simple Lie algebras.
\newblock {\em Lie groups and invariant theory}, 85--104, Amer. Math. Soc. Transl. Ser. 2, 213, {\em Amer. Math. Soc., Providence, RI}, 2005. 

\bibitem[FL]{FL}V. A. Fateev, S. L. Lukyanov.
\newblock Additional symmetries and exactly solvable models of two-dimensional conformal field theory.
\newblock {\em Sov. Sci. Rev. A. Phys.}, 15:1--117, 1990.

\bibitem[Fei]{F}B. L. Feigin.
\newblock Semi-infinite homology of Lie, Kac-Moody and Virasoro algebras.
\newblock {\em Uspekhi Mat. Nauk}, 39(2(236)):195--196, 1984. 

\bibitem[FF1]{FF90}B. L. Feigin, E. Frenkel.
\newblock Quantization of Drinfel'd-Sokolov reduction.
\newblock {\em Phys. Lett., B} 246(1--2):75--81, 1990.

\bibitem[FF2]{FF92}B. L. Feigin, E. Frenkel.
\newblock Affine Kac-Moody algebras at the critical level and Gel'fand-Dikii algebras. 
\newblock {\em Infinite analysis, Part A, B (Kyoto, 1991)}, 197--215, Adv. Ser. Math. Phys., 16, {\em World Sci. Publ., River Edge, NJ}, 1992. 

\bibitem[FS]{FS}B. L. Feigin, A. M. Semikhatov.
\newblock $\W^{(2)}_{n}$-algebras.
\newblock {\em Nuclear Phys., B} 698(3):409--449, 2004.

\bibitem[FOR]{FOR}J. M. Figueroa-O'Farrill, E. Ramos.
\newblock Classical $N=1$ $W$-superalgebras from Hamiltonian reduction.
\newblock {\em Comm. Math. Phys.}, 145(1):43--55, 1992.

\bibitem[Fre]{Fre}E. Frenkel.
\newblock Wakimoto modules, opers and the center at the critical level.
\newblock {\em Adv. Math.}, 195(2):279--327, 2004.

\bibitem[FBZ]{FBZ}E. Frenkel, D. Ben-Zvi.
\newblock Vertex algebras and algebraic curves. Second edition.
\newblock Mathematical Surveys and Monographs, 88.
\newblock {\em American Mathematical Society, Providence, RI}, 2004.

\bibitem[FKW]{FKW}E. Frenkel, V. G. Kac, M. Wakimoto. 
\newblock Characters and fusion rules for $W$-algebras via quantized Drinfel'd-Sokolov reduction.
\newblock {\em Comm. Math. Phys.}, 147(2):295--328, 1992.

\bibitem[Fu]{Fu}D. B. Fuchs.
\newblock Cohomology of infinite-dimensional Lie algebras.
\newblock Contemporary Soviet Mathematics. {\em Consultants Bureau, New York}, 1986.

\bibitem[Ha]{Ha}R. Hartshorne.
\newblock Algebraic geometry.
\newblock Graduate Texts in Mathematics, No.52. {\em Springer-Verlag, New York-Heidelberg}, 1977.

\bibitem[HD]{HD}R. Heluani, L. O. Rodr\'{i}guez D\'{i}az.
\newblock The Shatashvili-Vafa $G_{2}$ superconformal algebra as a quantum Hamiltonian reduction of $D(2,1;\alpha)$.
\newblock {\em Bull. Braz. Math. Soc. (N.S.)}, 46(3):331--351, 2015. 

\bibitem[Ho]{Ho}C. Hoyt.
\newblock Good gradings of basic Lie superalgebras.
\newblock {\em Israel J. Math.}, 192(1):251--280, 2012.

\bibitem[IMP]{IMP}K. Ito, J. O. Madsen, J. L. Petersen.
\newblock Free field representations of extended superconformal algebras.
\newblock {\em Nuclear Phys., B} 398(2):425--458, 1993.

\bibitem[Kac1]{K77}V. G. Kac.
\newblock Lie superalgebras
\newblock {\em Advances in Math.}, 26(1):8--96, 1977.

\bibitem[Kac2]{K90}V. G. Kac.
\newblock Infinite-dimensional Lie algebras. Third edition.
\newblock {\em Cambridge University Press, Cambridge}, 1990.

\bibitem[Kac3]{K98}V. G. Kac.
\newblock Vertex algebras for beginners. Second edition.
\newblock University Lecture Series, 10.
\newblock {\em American Mathematical Society, Providence, RI}, 1998.

\bibitem[KRW]{KRW}V. G. Kac, S.-S. Roan, M. Wakimoto.
\newblock Quantum reduction for affine superalgebras.
\newblock {\em Comm. Math. Phys.}, 241(2-3):307--342, 2003.

\bibitem[KW1]{KW1}V. G. Kac, M. Wakimoto.
\newblock Integrable highest weight modules over affine superalgebras and number theory.
\newblock {\em Lie theory and geometry}, 415--456, Progr. Math., 123, {\em Birkh\"auser Boston, Boston, MA}, 1994.

\bibitem[KW2]{KW2}V. G. Kac, M. Wakimoto.
\newblock Quantum reduction and representation theory of superconformal algebras.
\newblock {\em Adv. Math.}, 185(2):400--458, 2004.

\bibitem[KW3]{KW3} V. G. Kac, M. Wakimoto.
\newblock Corrigendum to: ``Quantum reduction and representation theory of superconformal algebras'' [Adv. Math. 185(2004)400--458].
\newblock {\em Adv. Math.}, 193(2):453--455, 2005. 

\bibitem[Kos]{Kos}B. Kostant.
\newblock Lie algebra cohomology and the generalized Borel-Weil theorem.
\newblock {\em Ann. of Math. (2)}, 74:329--387, 1961.

\bibitem[L]{L}H. Li.
\newblock Vertex algebras and vertex Poisson algebras.
\newblock {\em Commun. Contemp. Math.}, 6(1):61--110, 2004.

\bibitem[M]{M}I. M. Musson.
\newblock Lie superalgebras and enveloping algebras.
\newblock Graduate Studies in Mathematics, 131.
\newblock {\em American Mathematical Society, Providence, RI}, 2012.

\bibitem[P]{P}A. M. Polyakov.
\newblock Gauge transformations and diffeomorphisms.
\newblock {\em Internat. J. Modern Phys.}, A5(5):833--842, 1990.

\bibitem[Wang]{Wan}W. Wang.
\newblock Nilpotent orbits and $W$-algebras.
\newblock {\em Geometric representation theory and extended affine Lie algebras}, 71--105, Fields Inst. Commun., 59, {\em Amer. Math. Soc., Providence, RI}, 2011.

\bibitem[Watts]{Wat}G. M. T. Watts.
\newblock $W B_{n}$ symmetry, Hamiltonian reduction and $B(0, n)$ Toda theory.
\newblock {\em Nuclear Phys., B} 361(1):311--336, 1991.

\bibitem[Za]{Za}A. B. Zamolodchikov.
\newblock Infinite extra symmetries in two-dimensional conformal quantum field theory.
\newblock {\em Teoret. Mat. Fiz.}, 65(3):347--359, 1985. 

\bibitem[Zhu]{Zh}Y. Zhu.
\newblock Modular invariance of characters of vertex operator algebras.
\newblock {\em J. Amer. Math. Soc.}, 9(1):237--302, 1996.

\end{thebibliography}
\end{document}